\documentclass[12pt,a4paper]{amsart}
\usepackage{a4wide}
\usepackage{amsmath, amssymb, amsfonts,enumerate}
\usepackage[all]{xy}
\usepackage{amscd}
\usepackage{comment}
\usepackage{graphicx}
\usepackage{amssymb}
\usepackage{amsmath}
\usepackage{latexsym}
\usepackage{amsthm}
\usepackage{autograph,epic,latexsym,bezier,amsbsy,color,enumerate,amsfonts,amsmath,amscd,amssymb}
\usepackage[latin1]{inputenc}

\newtheorem{theorem}{Theorem}[section]
\newtheorem{lemma}[theorem]{Lemma}
\newtheorem{corollary}[theorem]{Corollary}
\newtheorem{proposition}[theorem]{Proposition}

\theoremstyle{definition}
\newtheorem{definition}[theorem]{Definition}
\newtheorem*{definition*}{Definition}

\theoremstyle{remark}
\newtheorem*{remark}{Remark}

\newtheorem{example}[theorem]{Example}

\numberwithin{equation}{section}

\newcommand {\N}{\mathbb{N}} 

\newcommand{\PP}{\mathcal{P}}
\newcommand{\CC}{\mathcal{C}}
\newcommand{\LL}{\mathcal{L}}

\newcommand{\GG}{\mathcal{G}}

\DeclareMathOperator{\M}{Mat}

\DeclareMathOperator{\Map}{Map}

\DeclareMathOperator{\Int}{Int}

\DeclareMathOperator{\Id}{Id}
\DeclareMathOperator{\Tr}{Tr}

\begin{document}
\title[On sofic monoids]{On sofic monoids} 

\author[T.Ceccherini-Silberstein]{Tullio Ceccherini-Silberstein}
\address{Dipartimento di Ingegneria, Universit\`a del Sannio, C.so
Garibaldi 107, 82100 Benevento, Italy}
\email{tceccher@mat.uniroma3.it}
\author[M.Coornaert]{Michel Coornaert}
\address{Institut de Recherche Math\'ematique Avanc\'ee \\
 UMR 7501, Universit\'e de Strasbourg et CNRS \\
                                                 7 rue Ren\'e-Descartes \\
 67000 Strasbourg, France  }
\email{coornaert@math.unistra.fr}
\date{\today}
\begin{abstract}
We  investigate a notion of soficity for monoids.
A group is sofic as a group if and only if it is sofic as a monoid.
All finite monoids, all commutative monoids, all free monoids, all cancellative one-sided amenable monoids, all multiplicative monoids of matrices over a field, and all monoids obtained by adjoining an identity element to a semigroup are sofic.
On the other hand, although the question of the existence of a non-sofic group remains open,
we prove that the bicyclic monoid is not sofic.
This shows that there exist finitely presented amenable inverse monoids that are non-sofic.
\end{abstract}

\maketitle

\section{Introduction}
\label{s:introduction}

Sofic groups were introduced at the end of the last century by M.~Gromov \cite{gromov-esav} and B.~Weiss \cite{weiss-sgds}.
The class of groups  they constitute is very large since it includes in particular all locally residually amenable groups and hence all linear groups.
Actually, the question whether or not every group is sofic remains open up to now although several experts in the field think that the answer to this question should be negative. 
Roughly speaking, a group is sofic when it can be well approximated by finite symmetric groups.
Sofic groups satisfy certain finiteness properties that are important in the theory of dynamical systems and operator algebras. For example, it is known that every sofic group
is surjunctive \cite{weiss-sgds}, hyperlinear \cite{elek-hyperlinearity}, and has stably finite group algebras whatever the ground field \cite{elek-stable-finiteness}.
For an introduction to the theory of sofic groups, the reader is referred to the excellent survey paper \cite{pestov-guide} or to \cite[Chapter 7]{book}.  
\par
The theme of soficity was fruitfully developed in several other directions: weakly-sofic groups 
\cite{glebsky-rivera}, linearly sofic groups \cite{arzhantseva-linear-sofic}, \cite{stolz}, sofic groupoids of measure-preserving transformations \cite{dykema-sofic-groupoid}, \cite{bowen-sofic-groupoid}, and sofic measure-preserving equivalence relations \cite{elek_sofic-equivalence}.
In each of these settings, the basic question of the existence of a non-sofic object remains still unanswered. 
\par
The goal of the present note is to investigate a notion of soficity for monoids, i.e., sets
equipped with a binary operation that is associative and admits an identity element.
With our definition, a group is sofic as a monoid if and only if it is sofic as a group. 
As every submonoid of a sofic monoid is sofic, this implies that every monoid that can be embedded into a sofic group is itself sofic.
Consequently, all free monoids, all cancellative one-sided amenable monoids, are sofic.
The class of sofic monoids is closed under  direct products, projective limits and inductive limits.
We shall also see that all finite monoids and all commutative monoids are sofic.  
As there exist finite monoids as well as commutative monoids that are not cancellative, 
this shows in particular that there are sofic monoids that cannot be embedded into groups.
On the other hand, we shall prove that the bicyclic monoid is non-sofic.
Thus there exist finitely presented amenable inverse monoids that are not sofic.

Finally, we shall present a graph theoretic characterization of soficity for finitely generated left-cancellative monoids in terms of approximability of their Cayley graphs by finite labeled graphs. This characterization is analogous to the one used by Weiss in \cite{weiss-sgds} for defining sofic groups.
  
\section{Background material}
\label{sec:background}

\subsection{Semigroups and monoids}
A \emph{semigroup} is a set equipped with an associative binary operation.
Unless stated otherwise, we will use  multiplicative notation for the binary operation on a semigroup. 
\par
Let $S$ be a semigroup.
\par
Given $s \in S$, we denote by $L_s$ and $R_s$ the left and right multiplication by $s$, that is,
the maps $L_s \colon S \to S$ and $R_s \colon S \to S$ defined by $L_s(t) = st$ and $R_s(t) = ts$ for all $t \in S$.
An element $s \in S$  is called \emph{left-cancellable} (resp. \emph{right-cancellable}) 
if the map $L_s$ (resp. $R_s$) is injective.
One says that an element $s \in S$ is cancellable if it is both left-cancellable and right-cancellable.
The semigroup $S$ is called \emph{left-cancellative} (resp. \emph{right-cancellative}, resp. \emph{cancellative}) if every element in $S$ is left-cancellable  (resp.  right-cancellable, resp. cancellable).
\par
Given semigroups $S_1$ and $S_2$, a map $\varphi \colon S_1 \to S_2$ is called a 
\emph{semigroup morphism} if it satisfies $\varphi(st) = \varphi(s) \varphi(t)$ for all $s,t \in S_1$.
\par
A \emph{subsemigroup} of a semigroup $S$ is a subset $T\subset S$ such that $s_1 s_2 \in T$ for all $s_1,s_2 \in T$.
\par
A semigroup $S$ is called an \emph{inverse semigroup} if, for every $s \in S$, there exists a unique element $x \in S$ such that $s = sxs$ and $x = xsx$.
\par
A \emph{monoid} is a semigroup admitting an identity element.
If $M$ is a monoid, we denote its identity element by $1_M$.
\par
Given two monoids $M_1$ and $M_2$, a semigroup morphism $\varphi \colon M_1 \to M_2$ is called a \emph{monoid morphism} if it satisfies $\varphi(1_{M_1}) = 1_{M_2}$.
A \emph{submonoid} of a monoid $M$ is a subsemigroup $N \subset M$ such that $1_M \in N$.  
\par 
Let $\PP$ be a property  of monoids (e.g., being finite).  
One says that a monoid $M$ is \emph{locally} $\PP$ if every finitely generated submonoid of $M$  satisfies  $\PP$.
One says that a monoid $M$ is \emph{residually} $\PP$ if, given any pair of distinct elements $s_1,s_2 \in M$, there exists a monoid $N$ satisfying $\PP$ and a monoid morphism
$\varphi \colon M \to N$ such that $\varphi(s_1) \not= \varphi(s_2)$.
One says that a monoid $M$ is \emph{locally embeddable} into the class of monoids satisfying $\PP$ (or, for short, \emph{locally embeddable into} $\PP$) if, for every finite subset $K \subset M$, there exists a monoid $N$ satisfying $\PP$ and a map
$\varphi \colon M \to N$ satisfying the following properties:
(1) the restriction of $\varphi$ to $K$ is injective, (2) for all $k_1,k_2 \in K$, one has $\varphi(k_1 k_2) = \varphi(k_1) \varphi(k_2)$, (3) $\varphi(1_M) = 1_N$ 
(note that $\varphi$ is not required to be globally injective nor to be a semigroup morphism).

\begin{proposition}
\label{p:residually-LEP}
Let $\PP$ be a property of monoids.
Suppose that any finite product of monoids satisfying $\PP$ also satisfies $\PP$.
Then every locally residually $\PP$ monoid is locally embeddable into $\PP$.
\end{proposition}

\begin{proof}
Suppose that $M$ is a locally residually $\PP$ monoid and $K \subset M$ is a finite subset.
Denote by $T$ the monoid generated by $K$.
Let $D := \{\{s,t\} : s,t \in K \text{ and } s \not= t\}$.
As $T$ is residually $\PP$,  for each $d = \{s,t\} \in D$, there exist a monoid $R_d$ satisfying $\PP$  with  a monoid morphism $\psi_d \colon T \to R_d$ such that $\psi_d(s) \not= \psi_d(t)$.
By our hypothesis, the product monoid $P := \prod_{d \in D} R_d$ satisfies $\PP$.
The product monoid morphism $\psi := \Pi_{d \in D} \psi_d \colon T \to P$ is injective on $K$.
By extending arbitrarily $\psi$ to $M$, we get a map $\varphi \colon M \to P$ that is injective on $K$, and such that $\varphi(k_1 k_2) = \varphi(k_1) \varphi(k_2)$ for all $k_1,k_2 \in K$ and $\varphi(1_M) = 1_P$.
This shows that $M$ is locally embeddable into $\PP$.
\end{proof}

A monoid that is locally embeddable into the class of finite monoids is called an \emph{LEF-monoid}. As a product of finitely many finite semigroups is finite, we deduce from 
Proposition~\ref{p:residually-LEP} the following: 

\begin{corollary}
\label{c:loc-rf-are-LEF}
Every locally residually finite monoid is an LEF-monoid.
In particular, every residually finite monoid and every locally finite monoid is an 
LEF-monoid.
\qed
\end{corollary}

\subsection{Symmetric monoids and the Hamming metric} 
Let  $X$ be a set.
We denote by $\Map(X)$ the symmetric monoid of $X$, i.e., the set consisting of all maps
$f \colon X \to X$   with the composition of maps as the monoid operation.
The identity element of the symmetric monoid $\Map(X)$ is the identity map $\Id_X \colon X \to X$.
\par
Suppose  that  $X$ is a non-empty finite set.
The \emph{Hamming metric}   $d_X^{\text{Ham}}$  on 
 $\Map(X)$ is the metric defined by 
\[
 d_X^{\text{Ham}}(f,g) := \frac{1}{|X|} |\{x \in X : f(x) \not= g(x)\}|
 \] 
 for all $f,g \in \Map(X)$ (we use $|\cdot|$ to denote cardinality of finite sets). Note that
 $0 \leq  d_X^{\text{Ham}}(f,g) \leq 1$ for all $f,g \in \Map(X)$.
\par
Suppose now that  $X_1, X_2, \ldots, X_n$ is a finite sequence of  non-empty finite sets.  
Consider the Cartesian product $X = \prod_{1 \leq i \leq n} X_i$ and the natural semigroup   morphism  $\Phi \colon \prod_{1 \leq i \leq n} \Map(X_i) \to \Map(X)$ given by
$$
\Phi(f)(x) = (f_1(x_1),  \ldots, f_n(x_n))
$$
for all $f = (f_i)_{1 \leq i\leq n} \in \prod_{1 \leq i \leq n}\Map(X_i)$ and $x = (x_i)_{1 \leq i\leq n} \in X$.
 
\begin{proposition}
With the above notation, one has
\begin{equation}
\label{e;hamming-product}
 d_X^{\text{Ham}}(\Phi(f),\Phi(g)) = 1 - \prod_{1 \leq i \leq n} \left(1 - d_{X_i}^{\text{Ham}}(f_i,g_i)\right)
\end{equation}
for all $f = (f_i)_{1 \leq i\leq n}$ and $g = (g_i)_{1 \leq i \leq n}$ in $\prod_{1 \leq i \leq n} \Map(X_i)$.
\end{proposition}

\begin{proof}
The formula immediately follows from the equality \[
 \{ x \in X : \Phi(f)(x) = \Phi(g)(x) \} = \prod_{1 \leq i \leq n} \{x_i \in X_i : f_i(x_i) = g_i(x_i) \}
 \]
 after taking cardinalities of both sides.
  \end{proof}

\subsection{Labeled graphs}
Let $\Sigma$ be a finite set. A \emph{$\Sigma$-labeled graph} is a pair $\GG = (V,E)$, where $V$ is the set of \emph{vertices} and $E \subset V \times \Sigma \times V$
is the set of (\emph{$\Sigma$-labeled}) \emph{edges}. For instance, if $\Sigma \subset M$ is a finite generating subset of a monoid $M$, the associated \emph{Cayley graph} $\CC(M,\Sigma)$ has vertex set
$V:= M$ and edge set $E:=\{(s,\sigma,s\sigma): s \in M, \sigma \in \Sigma\}$.

Let $\GG = (V,E)$ be a $\Sigma$-labeled graph.

One says that $\GG$ is \emph{finite} if $V$ is finite. 
If $V' \subset V$ is a subset of vertices, then the graph $\GG' = (V',E')$, where 
$E' = E \cap (V' \times \Sigma \times V')$, is called
the \emph{subgraph of $\GG$ induced} by $V'$.

Given $e=(u,\sigma,v) \in E$, one says that $\alpha(e):=u\in V$ (resp. $\lambda(e) := \sigma \in \Sigma$, resp. $\omega(e):=v \in V$) is the \emph{initial vertex} (resp. \emph{label}, resp. \emph{terminal vertex}) of the edge $e$. An edge $e \in E$ such that
$\alpha(e) = \omega(e)$ is called a \emph{loop}. Note that in $\GG$ one may have \emph{multiple
edges} that is, distinct edges $e_1$ and $e_2$ satisfiyng $\alpha(e_1) = \alpha(e_2)$ and
$\omega(e_1) = \omega(e_2)$.

A \emph{path} of \emph{length} $n$ in $\GG$ is a sequence $\pi = (e_1,e_2, \ldots, e_n)$ of edges such that $\omega(e_i) = \alpha(e_{i+1}$ for $i=1,2,\ldots,n-1$; one then says that $\pi$ \emph{connects} $\alpha(e_1)$ to $\omega(e_n)$. Given $r \geq 0$ and $u \in V$, the \emph{ball of radius $r$ centered at $u$} is the
set $B_r^\GG(u)$ of all vertices $v \in V$ for which there exists a path connecting $u$ to $v$ of length $\leq r$.
A \emph{pointed} $\Sigma$-labeled graph is a pair $(\GG,v_0)$ where 
$\GG = (V,E)$ is a $\Sigma$-labeled graph and $v_0 \in V$ is a distinguished vertex. 
We shall regard the subgraph of $\GG$ induced by any ball $B_r^\GG(u)$ as a $\Sigma$-labeled graph pointed at its center $u$.

Finally, given another $\Sigma$-labeled graph $\GG' = (V',E')$, a \emph{label
graph isomorphism} from $\GG$ to $\GG'$ is a bijective map $\psi \colon V \to V'$ such that
$(\psi(u),\sigma, \psi(v)) \in E'$ for all $(u,\sigma,v) \in E$ and $(\psi^{-1}(u'),\sigma, \psi^{-1}(v')) \in E$ for all $(u',\sigma,v') \in E'$. If, in addition, $(\GG,v_0)$ and $(\GG',v_0')$
are pointed, we say that a label graph isomorphism $\psi$ from $\GG$ to $\GG'$ is \emph{pointed} provided $\psi(v_0) = v_0'$.

\section{Sofic monoids}
\label{sec:sofic}

 \begin{definition}
\label{d:almost-momorphism}
Let $M$ be a monoid, $K \subset M$   and $\varepsilon, \alpha > 0$. 
Let $N$ be a monoid equipped with a metric $d$.
\par
A map $\varphi \colon M \to N$ is called a $(K,\varepsilon)$-\emph{morphism}   if it satisfies 
\[
 d (\varphi(k_1k_2),\varphi(k_1)\varphi(k_2)) \leq \varepsilon \quad \text{for all  }k_1,k_2 \in K
 \]
 and 
 \[
 d(\varphi(1_M),1_N) \leq \varepsilon.
 \]
\par
A map $\varphi \colon M \to N$
is said to be $(K,\alpha)$-\emph{injective} if it satisfies
 \[
 d(\varphi(k_1),\varphi(k_2)) \geq \alpha
 \] 
 for all distinct $k_1, k_2 \in K$.
\end{definition}

If $X$ is   a non-empty finite set,
  we equip its symmetric monoid $\Map(X)$ with its Hamming metric.  

\begin{definition}
\label{d:sofic-semigroup}
A monoid $M$ is called \emph{sofic} 
if it satisfies the following condition:
 for every finite subset $K \subset M$ and every $\varepsilon > 0$,
there exist a non-empty finite set $X$ and a
$(K,1 - \varepsilon)$-injective  $(K,\varepsilon)$-morphism 
$\varphi \colon M \to \Map(X)$.
\end{definition}

 \begin{proposition}
\label{p:equiv-def-sofic}
Let $M$ be a monoid.
Then the following conditions are equivalent:
\begin{enumerate}[\rm (a)]
\item
$M$ is sofic;
\item
for every $0 < \alpha < 1$,  for every finite subset $K \subset M$ and every $\varepsilon > 0$,
there exist a non-empty finite set $X$ and a
$(K,\alpha)$-injective  $(K,\varepsilon)$-morphism 
$\varphi \colon M \to \Map(X)$.
\item 
there exists   $0 < \alpha < 1$ such that,
 for every finite subset $K \subset M$ and every $\varepsilon > 0$,
there exist   a non-empty finite set $X$ and a
$(K,\alpha)$-injective  $(K,\varepsilon)$-morphism 
$\varphi \colon M \to \Map(X)$.
\end{enumerate}
\end{proposition}

\begin{proof}
Let $0 < \alpha < 1$, $K \subset M$ a finite subset  and $\varepsilon > 0$.
Choose $\varepsilon' > 0$ small enough so that  $\alpha \leq 1 - \varepsilon'$ and $\varepsilon' \leq \varepsilon$.
If $M$ is sofic,  we can find a non-empty finite set $X$ and a $(K,1 - \varepsilon')$-injective 
$(K,\varepsilon')$-morphism $\varphi \colon M \to \Map(X)$. Then $\varphi$ is a
$(K,\alpha)$-injective 
$(K,\varepsilon)$-morphism. This shows that (a) implies (b).
\par
Condition (b) trivially implies (c).
\par
To complete the proof, it suffices to show that (c) implies (a).
We use the technique of ``amplification"
(see for example \cite[Theorem~3.5]{pestov-guide}, \cite[Proposition~3.4]{glebsky-rivera}). 
Suppose that (c) is satisfied for some $0 <  \alpha < 1$.
Let $K \subset M$ be a finite subset and $\varepsilon > 0$.
Choose an integer $n \geq 1$ large enough so that
 \begin{equation}
\label{e:condition-on-n-alpha}
1 - \left(1 - \alpha\right)^n \geq 1 - \varepsilon
\end{equation}
and then $\varepsilon' > 0$ such that
 \begin{equation}
\label{e:condition-on-epsilon-n}
1 - \left(1 - \varepsilon'\right)^n \leq  \varepsilon.
\end{equation}
By (c), there exist a non-empty finite set $X$ and a map
$\varphi \colon M \to \Map(X)$ that is a $(K,\alpha)$-injective $(K,\varepsilon')$-morphism.
\par
Consider the diagonal monoid morphism $\Delta \colon \Map(X) \to \Map(X^n)$ defined by
\[
\Delta(f)(x_1,\dots,x_n) := (f(x_1),  \ldots, f(x_n))
\]
for all $f  \in  \Map(X)$ and $ (x_1,\dots,x_n) \in X^n$.
  Then the composite map $\psi := \Delta \circ \varphi  \colon M \to \Map(X^n)$ satisfies, for all distinct $k_1,k_2 \in K$,
\begin{align*}
d_{X^n}^{\text{Ham}}(\psi(k_1),\psi(k_2))
&= d_{X^n}^{\text{Ham}}(\Delta (\varphi(k_1)),\Delta(\varphi(k_2))) \\
&= 1 -   \left(1 - d_{X}^{\text{Ham}}(\varphi(k_1),\varphi(k_2))\right)^n  
&& \text{(by \eqref{e;hamming-product})} \\
& \geq 1 - \left(1 - \alpha\right)^n \\ 
&\geq 1 - \varepsilon && \text{(by \eqref{e:condition-on-n-alpha})}. 
\end{align*} 
 \par
On the other hand,   for all $k_1,k_2 \in K$, 
\begin{align*}
d_{X^n}^{\text{Ham}}(\psi(k_1k_2),\psi(k_1)\psi(k_2))
&= d_{X^n}^{\text{Ham}}(\Delta(\varphi(k_1k_2)),\Delta(\varphi(k_1))\Delta(\varphi(k_2)))  \\
 &= d_{X^n}^{\text{Ham}}(\Delta(\varphi(k_1k_2)),\Delta(\varphi(k_1) \varphi(k_2))) \\
  &= 1 -   \left(1 - d_{X}^{\text{Ham}}(\varphi(k_1 k_2),\varphi(k_1) \varphi(k_2))\right)^n  
&& \text{(by \eqref{e;hamming-product})} \\
&\leq 1 - \left(1 - \varepsilon'\right)^n \\
& \leq \varepsilon 
&& \text{(by \eqref{e:condition-on-epsilon-n})}.
\end{align*} 
Moreover, we also have
\begin{align*}
d_{X^n}^{\text{Ham}}(\psi(1_M),\Id_{X^n})
&= d_{X^n}^{\text{Ham}}(\Delta(\varphi(1_M)),\Id_{X^n}) 
 \\
 &= d_{X^n}^{\text{Ham}}(\Delta(\varphi(1_M)),\Delta(\Id_X)) \\
   &= 1 -   \left(1 - d_{X}^{\text{Ham}}(\varphi(1_M),\Id_X)\right)^n  
&& \text{(by \eqref{e;hamming-product})} \\
&\leq 1 - \left(1 - \varepsilon'\right)^n \\
& \leq \varepsilon 
&& \text{(by \eqref{e:condition-on-epsilon-n})}.
\end{align*}
We deduce that  the map  $\psi \colon M \to \Map(X^n)$ is a $(K,\varepsilon)$-injective   
$(K,\varepsilon)$-morphism. 
This shows that (c) implies (a).
\end{proof}

\begin{proposition}
\label{p:sofic-groups-monoids}
Let $G$ be a group.
Then $G$ is sofic as a group if and only if it is sofic as a monoid.
\end{proposition}

\begin{proof}
The fact that any group that is sofic as a monoid is also sofic as a group
is clear if we compare our Definition~\ref{d:almost-momorphism} and 
Definition~\ref{d:sofic-semigroup} above with Definition~1.1 and Definition~1.2   in 
\cite{elek-on-sofic-groups}. The converse implication, namely that any group that is sofic as a group is also sofic as a monoid, follows from  our definitions and Lemma~2.1 in \cite{elek-on-sofic-groups}.
 \end{proof}

\begin{proposition}
\label{p;subgroups-sofic}
Every submonoid of a sofic monoid is sofic.
\end{proposition}

\begin{proof}
Let $M$ be a sofic monoid and $N$  a submonoid of $M$. 
Fix a finite subset $K \subset N$ and
$\varepsilon > 0$. As $M$ is sofic, there exists a non-empty finite  set $X$ and  a
$(K,1- \varepsilon)$-injective  $(K,\varepsilon)$-morphism $\varphi\colon M \to \Map(X)$. Then the restriction map
$\varphi\vert_N \colon N \to \Map(X)$ is a 
$(K,1 - \varepsilon)$-injective $(K,\varepsilon)$-morphism. 
This shows that the monoid $N$ is sofic.
\end{proof}

\begin{proposition}
\label{p:locally-sofic}
Every locally sofic monoid is sofic.
\end{proposition}

\begin{proof}
Let $M$ be a locally sofic monoid. Let $K \subset M$ be a finite subset and $\varepsilon > 0$.
Denote by $N$ the submonoid of $M$ generated by $K$. As $N$ is sofic, there exist a 
non-empty finite  set $X$ and
a $(K,1 - \varepsilon)$-injective $(K,\varepsilon)$-morphism $\psi \colon N \to \Map(X)$.
By extending arbitrarily $\psi$ to $M$,  
we get a $(K,1 - \varepsilon)$-injective  $(K,\varepsilon)$-morphism $\varphi \colon M \to \Map(X)$. 
This shows that $M$ is sofic.
 \end{proof}

\begin{proposition}
\label{p:prod-sofic}
Let $(M_i)_{i\in I}$ be a family of sofic monoids. 
Then  the  product monoid $M := \prod_{i \in I}M_i$ is also sofic.
\end{proposition}

\begin{proof}
For each $i \in I$, let  $\pi_i \colon M \to M_i$ denote the
projection  morphism.
Fix a finite subset $K \subset M$ and $\varepsilon > 0$.
Then there exists a finite subset
$J \subset I$ such that the projection 
$\pi_J   \colon M \to M_J := \prod_{j \in J}M_j$ is
injective on $K$. Choose a constant $0 < \eta < 1$  small enough so that
\begin{equation}
\label{e;eta-1}
1 - (1 - \eta)^{|J|} \leq \varepsilon
\end{equation}
and $\eta\leq \varepsilon$.
\par
 Since the monoid $M_j$ is sofic for each $j\in J$,
there exist a nonempty finite set $X_j$ and a $(\pi_j(K),1 - \eta)$-injective $(\pi_j(K), \eta)$-morphism
$\varphi_j \colon M_j \to \Map(X_j)$.
Consider the nonempty finite set $X := \prod_{j \in J}X_j$ and the map $\varphi \colon M \to \Map(X)$ defined by
$$
\varphi(m)(x) := (\varphi_j(m_j)(x_j))_{j \in J}
$$
for all $m = (m_i)_{i \in I} \in M$ and $x = (x_j)_{j \in J} \in X$.
For all $k= (k_i)_{i \in I},k' = (k_i')_{i \in I} \in K$, we have 
\begin{align*}
d_X^{\text{Ham}}(\varphi(kk'), \varphi(k)\varphi(k'))  & =
1 - \prod_{j \in J} \left(1 - d_{X_j}^{\text{Ham}}(\varphi_j(k_jk_j'),\varphi_j(k_j)\varphi_j(k_j'))\right) 
&& \text{(by \eqref{e;hamming-product})} \\  
& \leq 1 - (1 - \eta)^{|J|} \\
& \leq \varepsilon && \text{(by \eqref{e;eta-1})}.
\end{align*}
We also have
\begin{align*}
d_X^{\text{Ham}}(\varphi(1_M), \Id_X)  & =
1 - \prod_{j \in J} \left(1 - d_{X_j}^{\text{Ham}}(\varphi_j(1_{M_j}),\Id_{X_j})\right) 
&& \text{(by \eqref{e;hamming-product})} \\  
& \leq 1 - (1 - \eta)^{|J|} \\
& \leq \varepsilon && \text{(by \eqref{e;eta-1})}.
\end{align*}
On the other hand, if $k$ and $k'$  are distinct elements in $K$, then there exists $j_0 \in J$ such that $k_{j_0} \neq k_{j_0}'$. This 
implies 
\begin{align*}
d_X^{\text{Ham}}(\varphi(k), \varphi(k')) 
&= 1 - \prod_{j \in J}\left(1 - d_{X_j}^{\text{Ham}}(\varphi_j(k_j), \varphi_j(k_j'))\right) 
&& \text{(by \eqref{e;hamming-product})} \\  
&\geq 1 - \left(1 - d_{X_{j_0}}^{\text{Ham}}(\varphi_{j_0}(k_{j_0}),\varphi_{j_0}(k_{j_0}'))\right) \\ 
&  \geq 1 -  \eta 
&& \text{(since $\varphi_{j_0}$ is  $(K_{j_0},1 - \eta)$-injective)} \\ 
& \geq 1 - \varepsilon && \text{(since $\eta \leq \varepsilon$)}. 
\end{align*}
It follows that $\varphi$ is a $(K,1 - \varepsilon)$-injective $(K,\varepsilon)$-morphism. 
This shows that $M$ is sofic.
\end{proof}

\begin{corollary}
Let $(M_i)_{i \in I}$ be a family of sofic monoids.
Then their direct sum
$M = \oplus_{i \in I} M_i$ is also sofic.
\end{corollary}

\begin{proof}
This immediately follows from Proposition~\ref{p;subgroups-sofic} and Proposition~\ref{p:prod-sofic}
since $M = \oplus_{i \in I} M_i$ is a submonoid of the product monoid $\prod_{i \in I} M_i$. 
\end{proof}
\begin{corollary}
\label{c:projective-sofic}
If a monoid $M$ is the limit of a projective system of sofic monoids then $M$ is sofic.
\end{corollary}

\begin{proof}
If $M$  is the limit of a projective system of monoids $(M_i)_{i \in I}$ then $M$ is a submonoid of the product $\prod_{i \in I} M_i$.
Thus, it follows from Proposition~\ref{p;subgroups-sofic} and Proposition~\ref{p:prod-sofic} that $M$ is sofic 
if every $M_i$, $i \in I$, is sofic.
\end{proof}

Recall that an inductive system of monoids, denoted $(M_i,\psi_{ji})$, consists of the following data: a directed set $I$, a family $(M_i)_{i \in I}$ of monoids and, for all $i, j \in I$ such that $i < j$, a monoid homomorphism $\psi_{ji} \colon M_i \to M_j$. Moreover these homomorphisms must satisfy $\psi_{ii} = \Id_{M_i}$ and $\psi_{kj} \circ \psi_{ji} = \psi_{ki}$ for all $i < j < k$ in $I$. Then the associated limit is the monoid $M:= (\coprod_{i \in I} M_i)/\sim$ (here $\coprod$ denotes a disjoint union of sets) where $\sim$ is the equivalence relation on $\coprod_{i \in I} M_i$ defined as follows: for $x_i \in M_i$ and $x_j \in M_j$, $i,j \in I$, one has $x_i \sim x_j$ provided there exists $\ell \in I$ such that $i \leq \ell$ and $j \leq \ell$ and $\psi_{\ell i}(x_i) = \psi_{\ell j}(x_j)$ in $M_\ell$. Denoting by $[x_i]:= \{x_j \in M_j: x_i \sim x_j, j \in I\} \in M$ the equivalence class of $x_i \in M_i$, the multiplication in $M$ is defined by $[x_i][y_j]:= [\psi_{\ell i}(x_i)\psi_{\ell j}(y_j)]$, where $\ell \in I$ is such that $i \leq \ell$ and $j \leq \ell$. This means that the canonical map $x_i \mapsto [x_i]$ is a monoid homomorphism from
$M_i$ into $M$ for all $i \in I$. 

\begin{proposition}
\label{p:inductive-sofic}
If a monoid $M$ is the limit of an inductive system of sofic monoids then $M$ is sofic.
\end{proposition}
\begin{proof}
Let $(M_i,\psi_{ji})$ be an inductive system of sofic monoids and denote by $M$ its limit. 
Let $K \subset M$ ba a finite set and $\varepsilon > 0$.
Let us set $H:= K \cup K^2$. For every $h \in H$ we can find $i = i(h) \in I$ and $h_i \in M_i$ such that $h = [h_i]$. Let $\kappa \in I$ be such that $i(h) \leq \kappa$ for all $h \in H$. 
We then set $h_{\kappa}:= \psi_{\kappa i}(h_i) \in M_{\kappa}$ for all $h \in H$.
Thus $h = [h_{\kappa}]$ for all $h \in H$. Let now $k,k' \in K$. We have $kk' \in H$ and $kk'= [k_{\kappa}][k'_{\kappa}] = [k_{\kappa} k'_\kappa]$. Thus we can find $j = j(k,k') \in I$ such that $\kappa \leq j$ and $\psi_{j \kappa}((kk')_\kappa) = \psi_{j \kappa}(k_\kappa k'_\kappa)$ in $M_j$. 
Let $\ell \in I$ be such that $j(k,k') \leq \ell$ for all $k,k' \in K$ and set
$h_\ell:= \psi_{\ell \kappa}(h_\kappa) \in M_\ell$ for all $h \in H$.
Again, we have $h = [h_\ell]$ for all $h \in H$.  Also note that $(kk')_\ell = k_\ell 
k'_\ell$ for all $k,k' \in K$. Finally, we set $K_\ell:= \{k_\ell: k \in K\}\subset M_\ell$.

Since $M_\ell$ is sofic, we can find a non-empty finite set $X$ and a $(K_\ell,1-\varepsilon)$-injective $(K_\ell,\varepsilon)$-morphism $\varphi \colon M_\ell \to \Map(X)$. We then define a map $\overline{\varphi} \colon M \to \Map(X)$ by setting
\[
\overline{\varphi}(s):= \begin{cases} \varphi(s_\ell) & \mbox{ if } s \in H\\
\Id_X & \mbox{ otherwise.}
\end{cases}
\]
Let now $k,k' \in K$. We then have
\[
d_X^{\text{Ham}}(\overline{\varphi}(kk'),\overline{\varphi}(k)\overline{\varphi}(k')) =  d_X^{\text{Ham}}(\varphi(k_\ell k'_\ell), \varphi(k_\ell)\varphi(k'_\ell)) < \varepsilon.
\]
Suppose now that $k \neq k'$. Then $k_ell \neq k'_\ell$ so that
\[
d_X^{\text{Ham}}(\overline{\varphi}(k),\overline{\varphi}(k')) =  d_X^{\text{Ham}}(\varphi(k_\ell), \varphi(k'_\ell)) \geq 1 - \varepsilon.
\]
This shows that $\overline{\varphi}$ is a $(K,1-\varepsilon)$-injective $(K,\varepsilon)$-morphism.
It follows that $M$ is sofic.
\end{proof}

Recall that a directed family of submonoids of a monoid $M$ is a family $(M_i)_{i \in I}$ of submonoids of $M$ such that for all $i,j \in I$ there exists $k \in I$ such that $M_i \cup M_j \subset M_k$. Note that setting $i < j$ whenever $M_i \subset M_j$, we have that 
a directed family of submonoids yields an inductive system $(M_i,\psi_{ji})$ where
$\psi_{ji}$ is the inclusion morphism $M_i \hookrightarrow M_j$ for $i < j$.

\begin{corollary}
\label{c:increasing-union-sofic}
If a monoid $M$ is the union of a directed family of sofic submonoids then $M$ is also sofic.
\end{corollary}
\begin{proof}
The monoid $M$ is then the limit of the associated inductive system.
\end{proof}

\begin{remark}
Since every monoid is the union of the directed family of its finitely generated submonoids, from
Corollary \ref{c:increasing-union-sofic} we recover the fact
that every locally sofic monoid is sofic (cf. Proposition \ref{p:locally-sofic}).
\end{remark}

\begin{proposition}
\label{p:locally-emb-sofic}
Every monoid that is locally embeddable into the class of sofic monoids is itself sofic.
\end{proposition}

\begin{proof}
Let $M$ be a  monoid that is locally embeddable into the class of sofic monoids. 
Let $K \subset M$ be a finite subset and $\varepsilon > 0$.
We want to show that there exist a non-empty finite set $X$ and
a $(K,1 - \varepsilon)$-injective $(K,\varepsilon)$-morphism $\varphi \colon M \to \Map(X)$.
Without loss of generality, we may assume $1_M \in K$. 
By definition of local embeddability, there exist a sofic monoid $N$ and a map $\psi \colon M \to N$
which is injective on $K$ and satisfies
\begin{equation}
\label{e:psi-local-morphism}
\psi(k_1 k_2) = \psi(k_1) \psi(k_2)
\quad \text{for all $k_1,k_2 \in K$}
\end{equation}
and $\psi(1_M) = 1_N$.
Let $K' := \psi(K)$. Since $N$ is sofic, there exist a non-empty finite set $X$ and a 
$(K',1 - \varepsilon)$-injective $(K',\varepsilon)$-morphism $\varphi' \colon N \to \Map(X)$.
Then the composite map $\varphi' \circ \psi \colon M \to \Map(X)$ is   clearly  a
$(K,1 - \varepsilon)$-injective $(K,\varepsilon)$-morphism. 
This shows that $M$ is a sofic monoid.
\end{proof}

\begin{corollary}
\label{c:residually-sofic}
Every locally residually sofic monoid is sofic.
In particular, every residually sofic monoid is sofic.
\end{corollary}

\begin{proof}
This immediately follows from Proposition~\ref{p:residually-LEP}, Proposition~\ref{p:prod-sofic},
and Proposition~\ref{p:locally-emb-sofic}.
\end{proof}

\section{Examples of  sofic monoids}
\label{sec:finite-are-sofic}

\begin{proposition}
\label{t:finite-are-sofic}
Every finite monoid is sofic.
\end{proposition}

\begin{proof}
Any monoid $M$ is isomorphic to a submonoid of the symmetric monoid $\Map(M)$ via the Cayley map
$m \mapsto L_m$ that sends every $m \in M$ to the left multiplication by $m$.
As every submonoid of a sofic monoid is itself sofic by Proposition~\ref{p;subgroups-sofic},
 it suffices to prove that the symmetric monoid of any finite set is sofic. 
\par
Let $X$ be a finite set of cardinality $ |X| \geq 1$ 
and let $\alpha := 1/|X|$.
Then, for every $\varepsilon > 0$ and every $K \subset \Map(X)$,   the identity morphism 
$\Id_{\Map(X)} \colon  \Map(X) \to \Map(X)$ is a $(K,\alpha)$-injective $(K,\varepsilon)$-morphism.
 Thus, the monoid $\Map(X)$ satisfies condition (c) in Proposition~\ref{p:equiv-def-sofic}. 
This shows that   $\Map(X)$ is sofic.  
\end{proof}

From Proposition~\ref{t:finite-are-sofic}, Proposition~\ref{p:locally-emb-sofic}, and
Corollary~\ref{c:loc-rf-are-LEF}, we deduce the following result.

\begin{corollary}
\label{c:lef-are-sofic}
Every LEF-monoid is sofic.
In particular, every locally residually finite monoid, and hence every residually finite monoid and every locally finite monoid, is sofic.\qed
\end{corollary}

\begin{proposition}
\label{t:commutative-are-sofic}
Every commutative monoid is sofic.
\end{proposition}

\begin{proof}
This follows from  
Corollary~\ref{c:lef-are-sofic},
since, by a result of Mal'cev \cite{malcev} (see also \cite{lallement}, \cite{carlisle}), every   commutative semigroup is locally residually finite.
\end{proof}

\begin{corollary}
\label{c:free}
Every free monoid is sofic.
\end{corollary}

\begin{proof}
This follows from Corollary~\ref{c:lef-are-sofic}
since every free monoid is residually finite.
\end{proof}

\begin{corollary}
Let $K$ be a field and let $n \geq 1$ be an integer.
Then the multiplicative monoid $\M_n(K)$ formed by all $n \times n$ matrices with entries in $K$ is sofic.
\end{corollary}

\begin{proof}
This follows from  
Corollary~\ref{c:lef-are-sofic},
since, by a result of Mal'cev \cite{malcev-representation},
the multiplicative monoid $\M_n(K)$ is locally residually finite.
\end{proof}

\begin{remark}
In notes by Stallings \cite{stallings-notes},
it is shown that the field $K$ in Mal'cev result can be replaced by any commutative unital ring.
\end{remark}

\begin{proposition}
\label{p:canc-amen-are-sofic}
All cancellative one-sided amenable monoids are sofic.
\end{proposition}

\begin{proof}
It is known \cite[Corollary~3.6]{wilde-witz} that every cancellative left-amenable monoid is isomorphic to a submonoid of an amenable group. As the opposite semigroup of a right-amenable semigroup is left-amenable and every group is isomorphic to its opposite, we deduce that every cancellative right-amenable semigroup is also isomorphic to a submonoid of an amenable group. Thus, the result follows from Proposition~\ref{p;subgroups-sofic} and the fact that every amenable group is sofic as a group (see for instance \cite[Proposition 7.5.6]{book}) and hence sofic as a monoid by Proposition~\ref{p:sofic-groups-monoids}. 
\end{proof}

\begin{proposition}
\label{p:adjoining-identity-sofic}
Let $S$ be a semigroup and let $M = M(S)$ denote the monoid obtained from $S$ by adjoining an identity element. Then $M$ is sofic.
\end{proposition}

Recall that $M = S \cup \{1_M\}$, where $1_M \notin S$ satisfies $1_Ms=s1_M=s$ for all $s \in M$ and $S$ is a subsemigroup of $M$. When $S$ is a semigroup which is not a monoid then $M$ is the so-called \emph{minimal monoid} of $S$.

\begin{proof}[Proof of Proposition \ref{p:adjoining-identity-sofic}]
Let $K$ be a finite subset of $M$ and $\varepsilon > 0$.
Let $Y := K \cup K^2$ denote the subset of $M$ consisting of all elements that are in $K$ or may be written as the product of two elements in $K$.
Choose an arbitrary element $y_0 \notin Y$ and a finite set $Z$ disjoint from $Y \cup \{y_0\}$.
Let $X := Y \cup \{y_0\} \cup Z$ and consider the map $\varphi \colon M \to \Map(X)$ defined as follows.
We  take $\varphi(1_M) = \Id_X$ and,  for $s \in S$, define $\varphi(s) \in \Map(X)$ by
\[
\varphi(s)x =
\begin{cases}
s & \text{ if } s \in Y \text{ and } x \in Z\\
sx & \text{ if } s \in Y, x \in Y, \text{ and } sx \in Y\\
y_0 &\text{ otherwise}
\end{cases}
\]
for all $x \in X$.
\par  
For $k_1,k_2 \in K \setminus \{1_M\}$,
we have $k_1k_2 \in Y \setminus \{1_M\}$ so that 
\[
(\varphi(k_1)\varphi(k_2))(z) = \varphi(k_1)(k_2) = k_1 k_2 = \varphi(k_1k_2)(z).
\]
As $\varphi(k_1 k_2) = \varphi(k_1) \varphi(k_2)$ if $k_1 = 1_M$ or $k_2 = 1_M$, we deduce that
 \[
d_X^{\text{Ham}} (\varphi(k_1 k_2),  \varphi(k_1) \varphi(k_2))\leq 1 - \frac{|Z|}{|X|} 
\quad \text{for all  } k_1,k_2 \in K.
\]
 On the other hand, if $z \in Z$,
 we have $\varphi(1_M)(z) = z$ and $\varphi(k)(z) = k$
 for all $k  \in K \setminus \{1_M\}$.
It follows that
\[
d_X^{\text{Ham}}(\varphi(k_1),\varphi(k_2)) \geq \frac{|Z|}{|X|} \quad \text{for all distinct  } k_1,k_2 \in K.   
\]
Consequently, $\varphi$ is a $(K,1 - \varepsilon)$-injective $(K,\varepsilon)$-morphism for $|Z|$ large enough.
This shows that the monoid $M$ is sofic.
\end{proof}

\begin{remark}
One may rephrase Proposition \ref{p:adjoining-identity-sofic} by saying that every monoid $M$ in which the equation $x y = 1_M$ implies $x = y = 1_M$ is sofic.

Note also that, given any set $\Sigma$, the free monoid $M:=\Sigma^*$ is isomorphic to the monoid $M(S)$ obtained from  the free semigroup $S:=\Sigma^+ = \Sigma \Sigma^*$ by adjoining an identity element. Thus from Proposition \ref{p:adjoining-identity-sofic} we recover Corollary \ref{c:free}.
\end{remark}

\section{Non-soficity of the bicyclic monoid}
\label{sec:bicyclic}

 The \emph{bicyclic monoid} is the monoid $B$ given by the 
presentation $B = \langle p,q : pq = 1 \rangle$.
Every element $s \in B$ may be uniquely written in the form
$s = q^a p^b$, where $a = a(s)$ and $b = b(s)$ are non-negative integers.
 The bicyclic monoid may also be viewed as a submonoid of the symmetric monoid $\Map(\N)$ of the set of non-negative integers by regarding $p$ and $q$ as the maps respectively defined by
\[
p(n) =
\begin{cases}
n - 1 &\text{ if } n \geq 1 \\
0 & \text{ if } n = 0
\end{cases}
\quad
\text{and}
\quad
q(n) = n + 1
\quad
\text{for all }  n \in \N.
\]

\begin{theorem}
\label{t:bicyclic-not-sofic}
The bicyclic monoid $B = \langle p,q : pq = 1 \rangle$ is not sofic.
\end{theorem}

Let us first establish the following result.

\begin{lemma}
\label{e:fg-close-implies-gf-close}
Let $X$ be a non-empty finite set and let $f,g \in \Map(X)$. 
Then one has $d_X^{\text{Ham}}(fg, \Id_X) = d_X^{\text{Ham}}(gf,\Id_X)$.
\end{lemma}

\begin{proof}
By definition, the set $X_0:= \{x \in X: fg(x) = x\}$ satisfies $1 - |X_0|/|X| = d_X^{\text{Ham}}(fg, \Id_X)$. As the restriction of $g$ to $X_0$ is injective,
we have that $\vert g(X_0) \vert = \vert X_0 \vert$. 
Let now $x \in g(X_0)$ and let us denote by $y$ the unique element in $X_0$ such that $x = g(y)$.

Then we have $gf(x) = gfg(y) = g(y) = x$. We deduce that
$d_X^{\text{Ham}}(gf, \Id_X) \leq  1 - |g(X_0)|/|X|  = 1 - |X_0|/|X| = d_X^{\text{Ham}}(fg, \Id_X)$.
By exchanging $f$ and $g$, we similarly get
$d_X^{\text{Ham}}(fg, \Id_X) \leq d_X^{\text{Ham}}(gf, \Id_X)$.
It follows that $d_X^{\text{Ham}}(fg, \Id_X) = d_X^{\text{Ham}}(gf, \Id_X)$.
\end{proof}

\begin{remark}
One can give an alternative proof of Lemma \ref{e:fg-close-implies-gf-close}
based on the properties of the trace of square matrices.  
Indeed, consider the monoid monomorphism $\rho \colon \Map(X) \to \M_X(K)$, where $\M_X(K)$ is the multiplicative monoid of $X \times X$-matrices with entries in a field $K$, that sends each 
$f \in \Map(X)$ to the characteristic map of its graph $\{(f(x),x): x \in X\} \subset X \times X$.
In other word, $\Phi(f)$ is the $X \times X$ matrix with $0-1$ entries such that 
$\Phi(f)_{y,x} = 1$ if and only if $y = f(x)$.
Observe that $\Tr(\Phi(f))$ is the number of fixed points of $f$ if $K$ has characteristic $0$.
We deduce that
$$
d_X^{\text{Ham}}(fg, \Id_X) = 1 - \frac{\Tr(\Phi(fg))}{|X|} = 1 - \frac{\Tr(\Phi(f)\Phi(g))}{|X|} 
$$
for all $f,g \in \Map(X)$,
and hence $d_X^{\text{Ham}}(fg, \Id_X) = d_X^{\text{Ham}}(gf,\Id_X)$
since $\Tr(\Phi(f)\Phi(g)) = \Tr(\Phi(g)\Phi(f))$.
\end{remark}

\begin{proof}[Proof of Theorem \ref{t:bicyclic-not-sofic}]
Let  $K := \{1,p,q,qp\}$ and $0 < \varepsilon < \dfrac{1}{5}$.
Suppose that $X$ is a non-empty finite set and that $\varphi \colon B \to \Map(X)$ 
is a $(K,1 - \varepsilon)$-injective $(K,\varepsilon)$-morphism.
Consider the maps  $f:=\varphi(p)$ and $g:= \varphi(q)$.
 We then have
\begin{align*}
d_X^{\text{Ham}}(fg,\Id_X) &= d_X^{\text{Ham}}(\varphi(p)\varphi(q),\Id_X) \\
&\leq d_X^{\text{Ham}}(\varphi(pq),\Id_X) + d_X^{\text{Ham}}(\varphi(pq),\varphi(p)\varphi(q)) 
&& \text{(by the triangle inequality)} \\
&= d_X^{\text{Ham}}(\varphi(1_B),\Id_X) + d_X^{\text{Ham}}(\varphi(pq),\varphi(p)\varphi(q)) 
&& \text{(since $pq = 1_B$)} \\
&\leq 2\varepsilon
&& \text{(since $\varphi$ is a $(K,\varepsilon)$-morphism)}. \\
\end{align*}
Applying Lemma~\ref{e:fg-close-implies-gf-close}, we obtain
\begin{equation}
\label{e:gf-close-id}
d_X^{\text{Ham}}(gf,\Id_X) \leq 2 \varepsilon.
\end{equation}
Finally,  using again the triangle inequality, we get
\begin{align*}
d_X^{\text{Ham}}(\varphi(qp),\varphi(1_B)) 
&\leq d_X^{\text{Ham}}(\varphi(qp),gf) 
+ d_X^{\text{Ham}}(gf,\Id_X) + d_X^{\text{Ham}}(\varphi(1_B),\Id_X) \\ 
&\leq d_X^{\text{Ham}}(\varphi(qp),\varphi(q) \varphi(p)) + 2 \varepsilon + d_X^{\text{Ham}}(\varphi(1_B),\Id_X)
\quad  \text{(by \eqref{e:gf-close-id})} \\
&\leq 4 \varepsilon 
\qquad \qquad  \text{(since $\varphi$ is a $(K,\varepsilon)$-morphism)}.
\end{align*} 
This contradicts the fact that $\varphi$ is $(K,1 - \varepsilon)$-injective since $qp$ and $1_B$ are distinct elements of $K$  and
$4 \varepsilon < 1 - \varepsilon$.
Consequently, the monoid $B$ is not sofic.
\end{proof}

\begin{remark}
Since the bicyclic monoid $B$ is the homomorphic image of the free monoid based on $p$ and $q$, which is sofic (cf. Corollary \ref{c:free}), we deduce that the class of sofic monoids is not closed under taking images by monoid homomorphisms.
\end{remark}

\begin{corollary}
\label{c:amenable-non-sofic-exist}
There exist finitely presented amenable inverse monoids that are not sofic.
\end{corollary}

\begin{proof}
It is known that the bicyclic monoid is an amenable inverse monoid
(see for example \cite[Example 2, page 311]{duncan-namioka}).
\end{proof}

\begin{remark}
It follows from Corollary~\ref{c:amenable-non-sofic-exist}
that we cannot remove the cancellativity hypothesis in 
Proposition~\ref{p:canc-amen-are-sofic}.
\end{remark}

\begin{corollary}
\label{c:containig-bicyclic-implies-nonsofic}
Every monoid containing a submonoid isomorphic to the bicyclic monoid is non-sofic.
\end{corollary}

\begin{proof}
This immediately follows from Proposition~\ref{p;subgroups-sofic} and 
Theorem~\ref{t:bicyclic-not-sofic}.
\end{proof}

\begin{corollary}
\label{c:map-X-not-sofic}
Let $X$ be an infinite set. Then the symmetric monoid $\Map(X)$ is not sofic.
\end{corollary}

\begin{proof}
This follows from Corollary~\ref{c:containig-bicyclic-implies-nonsofic}
since $\Map(X)$ contains a submonoid isomorphic to $\Map(\N)$ and 
hence a submonoid isomorphic to the bicyclic monoid.
\end{proof}

\begin{corollary}
Let $K$ be a field and let $E$ be an infinite-dimensional vector space over $K$.
Let $\LL(E)$ denote the monoid consisting of all endomorphisms of $E$ with the composition of maps as the monoid operation. Then $\LL(E)$  is not sofic.
\end{corollary}

\begin{proof}
Let $(e_x)_{x \in X}$ be a basis of $E$.
Then the symmetric monoid $\Map(X)$ embeds into $\LL(E)$ via the map that sends each $f \in \Map(X)$ to the unique endomorphism $u$ of $E$ such that $u(e_x) = e_{f(x)}$ for all $x \in X$.
Thus $\LL(E)$ is not sofic by Corollary~\ref{c:map-X-not-sofic}.
\end{proof}

\begin{remark}
It would be interesting to give an example of a cancellative non-sofic monoid. Note that the
bicyclic monoid is neither left nor right-cancellative.
\end{remark}

\section{A graph-theoretic characterization of finitely generated left-cancellative sofic monoids}
\label{sec:graph-characterization}

In this section, we relate the notion of soficity for finitely generated monoids to a certain finiteness condition on their Cayley graphs.

Let $M$ be a finitely generated monoid and let $\Sigma \subset M$ be a finite generating subset of $M$. Let $\GG = (V,E)$ be a $\Sigma$-labeled graph. 
For $r \in \N$, we denote by $V(r) \subset V$ the set of vertices $v \in V$ for which there exists a $\Sigma$-labeled pointed graph isomorphism 
\begin{equation}
\label{eq:psi}
\psi_{v,r} \colon B_r(1_M) \to B_r(v),
\end{equation}
where $B_r(1_M) := B_r^{\CC(M,\Sigma)}(1_M)$ is the ball of radius $r$ centered at $1_M$ in the Cayley graph $\CC(M,\Sigma)$ and $B_r(v) := B_r^\GG(v)$ is the ball of radius $r$ centered at $v$ in $\GG$. Note that since $\psi_{v,r}$ is pointed we have
\begin{equation}
\label{eq:psi-pointed}
\psi_{v,r} (1_M) = v.
\end{equation}
Moreover, if it exists, $\psi_{v,r}$ is necessarily unique. 
Note also the obvious inclusions   
 \[
V = V(0)  \supset V(1) \supset V(2) \supset \dots.
\]

We say that the pair $(M,\Sigma)$ satisfies the \emph{Weiss condition} provided that,
for every $r \in \N$ and every $\delta > 0$, there exists a finite $\Sigma$-labeled graph 
$\GG = (V,E)$ satisfying
\begin{equation}
\label{eq:condition-weiss}
\vert V(r) \vert \geq (1-\delta) \vert V \vert.
\end{equation}

\begin{theorem}
\label{t:weiss-characterization}
Let $M$ be a finitely generated monoid and let $\Sigma \subset M$ be a finite generating subset of $M$. Then the following holds:
\begin{enumerate}[{\rm (1)}]
\item 
if $(M, \Sigma)$ satisfies the Weiss condition, then $M$ is sofic;
\item 
if $M$ is left-cancellative and sofic, then $(M, \Sigma)$ satisfies the Weiss condition. 
\end{enumerate}
\end{theorem}
\begin{proof}[Proof of Theorem \ref{t:weiss-characterization}.(1)]
Suppose that the pair $(M,\Sigma)$ satisfies the Weiss condition. In order to prove that
$M$ is sofic, let $K \subset M$ be a finite subset and $\varepsilon > 0$. We want to show that
there exist a non-empty finite set $X$ and a $(K,1 - \varepsilon)$-injective 
$(K,\varepsilon)$-morphism $\varphi \colon M \to \Map(X)$. 
Choose  $r_0 \in \N$ large enough so that  $K \cup K^2 \subset B_{r_0}(1_M)$.
Let $\GG = (V,E)$ be a finite $\Sigma$-labeled graph satisfying condition \eqref{eq:condition-weiss} for $r = 2r_0$ and $\delta = \varepsilon$. 
Note that \eqref{eq:condition-weiss}
now becomes
\begin{equation}
\label{eq:condition-weiss-2}
\vert V(r) \vert \geq (1-\varepsilon) \vert V \vert.
\end{equation}
Consider the map $\varphi \colon M \to \Map(V)$ defined by setting
\[
v^{\varphi(s)} = 
\begin{cases}
\psi_{v,r}(s) & \mbox{ if } s \in B_r(1_M) \mbox{ and } v \in V(r)\\
v & \mbox{ otherwise}
\end{cases}
\]
where $\psi_{v,r}$ is as in \eqref{eq:psi}. Let us show that $\varphi$ has the required
properties by taking $X:=V$.
\par
Let $k_1, k_2 \in K$ and $v \in V(r)$. 
We have 
\[
\begin{split}
v^{\varphi(k_1k_2)} & = \psi_{v,r}(k_1k_2)\\
& = \psi_{\psi_{v,r}(k_1),r}(k_2) \ \ \ \mbox{(by \eqref{eq:psi-2})}\\
& = (\psi_{v,r}(k_1))^{\varphi(k_2)}\\
& = (v^{\varphi(k_1)})^{\varphi(k_2)}\\
& = v^{\varphi(k_1)\varphi(k_2)}.
\end{split}
\]
This shows that the maps $\varphi(k_1k_2)$ and $\varphi(k_1)\varphi(k_2)$ coincide on $V(r)$.
Therefore, we deduce 
from \eqref{eq:condition-weiss-2}   that $d_V^{\text{Ham}}(\varphi(k_1k_2),\varphi(k_1)\varphi(k_2))
\leq \varepsilon$.
\par
Finally, suppose that $k_1$ and $k_2$ are distinct elements in $K$. 
We  have
\[
v^{\varphi(k_1)} = \psi_{v,r}(k_1) \neq \psi_{v,r}(k_2) = v^{\varphi(k_2)}
\]
since $\psi_{v,r}$ is injective.
From \eqref{eq:condition-weiss-2}, we then  deduce that 
$d_V^{\text{Ham}}(\varphi(k_1),\varphi(k_2))
\geq 1 -\varepsilon$. It follows that $\varphi$ is   $(K,1 - \varepsilon)$-injective.
\par 
Consequently, the monoid $M$ is sofic. This proves Theorem \ref{t:weiss-characterization}.(1).
\end{proof}

In order to prove Theorem \ref{t:weiss-characterization}.(2) let us first establish some auxiliary results.

\begin{lemma}
\label{semilemme}
Let $M$ be a sofic monoid, $K \subset M$ a finite subset, and $\varepsilon > 0$.
Then there exist a non-empty finite set $X$ and a 
$(K,1 - \varepsilon)$-injective $(K,\varepsilon)$-morphism 
$\varphi \colon M \to \Map(X)$ satisfying  $\varphi(1_M) = \Id_X$.
\end{lemma}

\begin{proof}
Since $M$ is sofic, we can find a non-empty finite set $X$ and a 
$(K,1 - \varepsilon/2)$-injective $(K,\varepsilon/2)$-morphism 
$\varphi' \colon M \to \Map(X)$. Let us show that the map $\varphi \colon M \to \Map(X)$
defined by setting $\varphi(s) = \varphi'(s)$ for all $s \in M \setminus \{1_M\}$ and
$\varphi(1_M) = \Id_X$ satisfies our requirements.
\par
Let $k_1,k_2 \in K$. 
If $k_1,k_2,k_1 k_2 \in K \setminus \{1_M\}$, we have 
\[
d_X^{\text{Ham}}(\varphi(k_1 k_2), \varphi(k_1)\varphi(k_2)) = d_X^{\text{Ham}}(\varphi'(k_1 k_2), \varphi'(k_1)\varphi'(k_2)) \leq \varepsilon/2 \leq \varepsilon.
\]
If $k_1 = 1_M$, then
$
d_X^{\text{Ham}}(\varphi(k_1 k_2), \varphi(k_1)\varphi(k_2))
= d_X^{\text{Ham}}(\varphi(k_2),   \varphi(k_2)) = 0
$.
Similarly, if $k_2 = 1_M$, then 
$
d_X^{\text{Ham}}(\varphi(k_1 k_2), \varphi(k_1)\varphi(k_2)) = d_X^{\text{Ham}}(\varphi(k_1), \varphi(k_1)) =0
$.
\par
Suppose now that $k_1 \not= 1_M \not= k_2$ and $k_1 k_2 = 1_M$.
Then we have 
\[
 d_X^{\text{Ham}}(\varphi(k_1 k_2), \varphi(k_1)\varphi(k_2))
 = d_X^{\text{Ham}}(\Id_X,\varphi'(k_1)\varphi'(k_2))
  \]
  which implies, by using the triangle inequality,
\begin{align*}
 d_X^{\text{Ham}}(\varphi(k_1 k_2), \varphi(k_1)\varphi(k_2)) 
 &\leq d_X^{\text{Ham}}(\varphi'(1_M),\Id_X) + d_X^{\text{Ham}}(\varphi'(1_M),\varphi'(k_1)\varphi'(k_2)) \\
 &\leq \varepsilon/2 + \varepsilon/2 = \varepsilon. 
\end{align*}
This shows that $\varphi$ is a $(K,\varepsilon)$-morphism.
To complete the proof, it remains only to check that $\varphi$ is $(1 - \varepsilon)$-injective.
 \par
Suppose that $k_1$ and $k_2$ are distinct elements in $K \setminus \{1_M\}$.
Then we have
\[
d_X^{\text{Ham}}(\varphi(k_1),  \varphi(k_2))
= d_X^{\text{Ham}}(\varphi'(k_1),  \varphi'(k_2))
\geq 1 - \varepsilon/2 \geq 1 - \varepsilon.
\]
On the other hand, if $1_M \in K$ and $1_M \not= k \in K$, then we have
\begin{align*}
d_X^{\text{Ham}}(\varphi(1_M),  \varphi(k))
&= d_X^{\text{Ham}}(\Id_X,  \varphi'(k)) \\
& \geq d_X^{\text{Ham}}(\varphi'(1_M),\varphi'(k)) - d_X^{\text{Ham}}(\varphi'(1_M),\Id_X)    
&& \text{(by the triangle inequality)} \\
&\geq 1 - \varepsilon/2 - \varepsilon/2 = 1 - \varepsilon. 
\end{align*}
Consequently, $\varphi$ is $(1 - \varepsilon)$-injective.
 \end{proof}
 
\begin{lemma}
\label{l:1} 
Let $M$ be a finitely generated monoid and let $\Sigma \subset M$ be a finite generating subset of $M$. Suppose that $M$ is left-cancellative and let $\GG = (V,E)$ be a $\Sigma$-labeled
graph. Let $r \in \N$ and $v \in V$. If $v \in V(2r)$ then $B_r(v) \subset V(r)$. 
\end{lemma}

\begin{proof}
Let $u \in B_r(v)$. 
We clearly have $B_r(u) \subset B_{2r}(v)$. 
Moreover, if $\psi_{v,2r} \colon B_{2r}(1_M) \to B_{2r}(v)$ is the unique  $\Sigma$-labeled pointed graph isomorphism, then setting $s:= \psi_{v,2r}^{-1}(u)$ we have $s \in B_r(1_M)$ so that $st \in B_{2r}(1_M)$ for all $t \in B_r(1_M)$. Since $M$ is left-cancellative the map $L_s \vert_{B_r(1_M)}$
yields a $\Sigma$-labeled pointed graph isomorphism from $B_r(1_M)$ onto $B_r(s)$. 
It follows that the composite map
\begin{equation}
\label{eq:psi-u-r}
\psi_{u,r} = \psi_{v,2r} \circ L_s \vert_{B_r(1_M)} \colon B_r(1_M) \to B_r(u)
\end{equation}
is also a $\Sigma$-labeled pointed graph isomorphism. 
We deduce that $u \in V(r)$.
This shows that $B_r(v) \subset V(r)$. 
\end{proof}

Note that the map $\psi_{u,r}$ defined in \eqref{eq:psi-u-r} explicitly is given by 
$\psi_{u,r}(t) = \psi_{v,2r}(st)$ for all $t \in B_r(1_M)$. 

\begin{lemma}
\label{l:2}
Let $M$ be a finitely generated monoid and let $\Sigma \subset M$ be a finite generating subset of $M$. Suppose that $M$ is left-cancellative and let $\GG = (V,E)$ be a $\Sigma$-labeled
graph. Let $r \in \N$, $s,t \in B_r(1_M)$ and $v \in V(2r)$. 
Then we have
\begin{equation}
\label{eq:psi-1}
\psi_{v,2r}(s) \in V(r)
\end{equation}
and
\begin{equation}
\label{eq:psi-2}
\psi_{v,2r}(st) = \psi_{\psi_{v,2r}(s),r}(t).
\end{equation}
\end{lemma}
\begin{proof}
Clearly $\psi_{v,2r}(s) \in B_r(v)$ so that \eqref{eq:psi-1} follows from Lemma \ref{l:1}.
To prove \eqref{eq:psi-2}, let us set $u:=\psi_{v,2r}(s)$.
Now, if $t = 1_M$, then \eqref{eq:psi-2} follows trivially since, keeping in mind \eqref{eq:psi-pointed}:
\[
\psi_{v,2r}(st) = \psi_{v,2r}(s) = \psi_{\psi_{v,2r}(s),r}(1_{M}) = \psi_{\psi_{v,2r}(s),r}(t).
\]
If $t \neq 1_M$, we can find $1 \leq r' \leq r$ and $\sigma_1, \sigma_2, \ldots, \sigma_{r'} \in \Sigma$ such that
$t = \sigma_1\sigma_2 \cdots \sigma_{r'}$. Consider the path
\[
\pi_1 = ((s,\sigma_1,s\sigma_1),(s\sigma_1,\sigma_2, s\sigma_1\sigma_2), \ldots, (s\sigma_1\sigma_2 \cdots \sigma_{r'-1}, \sigma_{r'}, st))
\]
and observe that it is entirely contained in $B_{2r}(1_M)$, since $s,t \in B_r(1_M)$. The morphism $\psi_{v,2r}$
maps $\pi_1$ into the path
\[
\begin{split}
\overline{\pi}_1 & = ((\psi_{v,2r}(s),\sigma_1, \psi_{v,2r}(s\sigma_1),(\psi_{v,2r}(s\sigma_1),\sigma_2, \psi_{v,2r}(s\sigma_1\sigma_2)), \ldots\\
& \ \ \ \ \ \ \ \ \ \ \ \ \ldots (\psi_{v,2r}(s\sigma_1\sigma_2 \cdots \sigma_{r'-1}), \sigma_{r'}, \psi_{v,2r}(st))).
\end{split}
\]
Recalling that $u=\psi_{v,2r}(s) \in B_r(v) \subset V(r)$, consider the inverse image of the
path $\overline{\pi}_1$ under the morphism $\psi_{u,r}$: 
\[
\pi_2 = ((1_{S},\sigma_1,\sigma_1),(\sigma_1,\sigma_2, \sigma_1\sigma_2), \ldots, (\sigma_1\sigma_2 \cdots \sigma_{r'-1}, \sigma_{r'}, t)).
\]
In particular, we have $\psi_{u,r}^{-1}(\psi_{v,2r}(st)) = t$, that is, $\psi_{v,2r}(st) =
\psi_{u,r}(t) = \psi_{\psi_{v,2r}(s),r}(t)$.
\end{proof}

\begin{proof}[Proof of Theorem \ref{t:weiss-characterization}.(2)]
Suppose that the monoid $M$ is left-cancellative and sofic. Let us show that the pair
$(M,\Sigma)$ satisfies the Weiss condition.
Fix $\delta > 0$ and $r \in \N$. 
Set $K := B_{2r+1}(1_M)$ and
\begin{equation}
\label{eq:delta-espilon}
\varepsilon := \left(|B_r(1_M)| \cdot |\Sigma| + |B_r(1_M)|^2\right)^{-1}\delta.
\end{equation}
Since $M$ is sofic, it follows from Lemma \ref{semilemme} 
that we can find a non-empty finite set $V$ and a
$(K,1 - \varepsilon)$-injective 
$(K, \varepsilon)$-morphism
$\varphi \colon M \to \Map(V)$ 
such that
\begin{equation}
\label{eq: identite}
\varphi(1_M) = \Id_V.
\end{equation}
In the sequel, it will be convenient to use the notation $v^f$ instead of $f(v)$ to denote the image of an element $v \in V$ by $f \in \Map(V)$.
\par
Consider the finite $\Sigma$-labeled graph $\GG = (V,E)$ whose edge set 
$E \subset V \times \Sigma \times V$ consists of all triples $(v,\sigma,v^{\varphi(\sigma)})$, 
where $v \in V$ and $\sigma \in \Sigma$.
 Note that $\GG$ may have loops and multiple edges.
 However,   
if $v \in V$ and $\sigma  \in \Sigma$ are fixed, then there exists a unique edge in $\GG$ with initial vertex $v$ and label $\sigma$.
\par
For each $v \in V$, denote by $\psi_v \colon B_{2r}(1_M) \to V$ the map defined by setting
\[
\psi_v(s) = v^{\varphi(s)}
\]
for all $s \in B_{2r}(1_M)$. 
Note that, by virtue of \eqref{eq: identite}, we have 
\begin{equation}
\label{e:psi-identite}
\psi_v(1_M) = v^{\varphi(1_M)} = v^{\Id_V} = v
\end{equation} 
for all $v \in V$.
\par
Denote by $V_0 \subset V$ the set of vertices $v \in V$ satisfying the following conditions:\\

\begin{enumerate}
\item[(*)] 
$\psi_v(s\sigma) = \psi_{\psi_v(s)}(\sigma)$ for all $s \in B_r(1_M)$ and $\sigma \in \Sigma$,\\
\item[(**)]
 $\psi_v(s) \neq \psi_v(t)$ for all distinct $s,t \in B_r(1_M)$.\\
\end{enumerate}

Let us show that $V_0 \subset V(r)$. Suppose that $v \in V_0$ and let $s \in B_r(1_M)$. 
We first show that
\begin{equation}
\label{eq:psi-surjective}
B_v(r) = \psi_v(B_r(1_M)).
\end{equation}
If $s = 1_M$, then we have $\psi_v(s) = \psi_v(1_M) = v \in B_r(v)$. 
If $s \neq 1_M$, then there exist $1 \leq r' \leq r$ and $\sigma_1, \sigma_2,
\ldots, \sigma_{r'} \in \Sigma$ such that $s = \sigma_1 \sigma_2 \cdots \sigma_{r'}$. 
Consider the sequence of edges
\[
\begin{split}
e_1 & := (v,\sigma_1, v^{\varphi(\sigma_1)}) = (v, \sigma_1, \psi_v(\sigma_1)),\\
e_2 & := (\psi_v(\sigma_1),\sigma_2, (\psi_v(\sigma_1))^{\varphi(\sigma_2)})\\
    & = (\psi_v(\sigma_1),\sigma_2, \psi_{\psi_v(\sigma_1)}(\sigma_2))\\
    & = (\psi_v(\sigma_1),\sigma_2, \psi_v(\sigma_1\sigma_2))) \ \ (\mbox{by } (*)),\\
\cdots & \cdots \\
e_{r'} & := (\psi_v(\sigma_1\sigma_2 \cdots \sigma_{r'-1}), \sigma_{r'}, (\psi_v(\sigma_1\sigma_2 \cdots \sigma_{r'-1}))^{\varphi(\sigma_{r'})})\\
& = (\psi_v(\sigma_1\sigma_2 \cdots \sigma_{r'-1}), \sigma_{r'}, \psi_{\psi_v(\sigma_1\sigma_2 \cdots \sigma_{r'-1})}(\sigma_{r'}))\\
& = (\psi_v(\sigma_1\sigma_2 \cdots \sigma_{r'-1}), \sigma_{r'}, \psi_v(\sigma_1\sigma_2 \cdots \sigma_{r'}))  \ \ (\mbox{by } (*)),\\
& = (\psi_v(\sigma_1\sigma_2 \cdots \sigma_{r'-1}), \sigma_{r'}, \psi_v(s)).
\end{split}
\]
The path $\pi = (e_1,e_2, \ldots, e_{r'})$ connects $v$ to $\psi_v(s)$ and has length $r' \leq r$.
This shows that $\psi_v(B_r(1_M)) \subset B_r(v)$. 
Conversely, suppose that  $u \in B_r(v)$. 
If $u = v$, then we have $u = v = \psi_v(1_M) \in \psi_v(B_r(1_M))$. If $u \neq v$ then
there exist $1 \leq r' \leq r$ and a path $((v,\sigma_1,v_1), (v_1,\sigma_2, v_2), \ldots, (v_{r'-1},\sigma_{r'}, u))$. Using both (*) and (**), we deduce that $u = \psi_v(s)$, where $s = \sigma_1\sigma_2 \ldots \sigma_{r'} \in B_r(1_M)$. This shows that $B_r(v) \subset \psi_v(B_r)$. We deduce
\eqref{eq:psi-surjective}.

Consider the map $\psi_{v,r} \colon B_r(1_M) \to B_r(v)$
obtained by restriction of
$\psi_v$ to $B_r(1_M)$. 
It follows from \eqref{eq:psi-surjective} and
condition (**) that $\psi_{v,r}$ is bijective. 
On the other hand, if $s \in B_r(1_M)$ and $\sigma \in \Sigma$, 
we deduce from  (*) that 
\[
(\psi_v(s),\sigma,\psi_v(s\sigma)) = (\psi_v(s),\sigma, \psi_{\psi_v(s)}(\sigma)) = (\psi_v(s),\sigma, (\psi_v(s))^{\varphi(\sigma)}) \in E.
\]
Moreover, by virtue of \eqref{e:psi-identite}, we have $\psi_{v,r}(1_M) = v$.
Consequently, $\psi_{v,r}$ is a $\Sigma$-labeled pointed graph isomorphism. 
This shows that $V_0 \subset V(r)$.
\par
Our next goal is to estimate from below the cardinality of $V_0$.
Let us start by estimating the cardinality of the set consisting of the vertices $v \in V$
satisfying condition (*).
Let $s \in B_r(1_M)$ and $\sigma \in \Sigma$. 
As $\varphi$ is a $(K,\varepsilon)$-morphism, we have
$d_V^{\text{Ham}}(\varphi(s\sigma),\varphi(s)\varphi(\sigma)) \leq \varepsilon$, so that there exists a subset
$V'(s,\sigma) \subset V$ with $|V'(s,\sigma)| \leq \varepsilon |V|$ such that
\[
\psi_v(s\sigma)  = v^{\varphi(s\sigma)} = v^{\varphi(s)\varphi(\sigma)} = (v^{\varphi(s)})^{\varphi(\sigma)} = (\psi_v(s))^{\varphi(\sigma)} = \psi_{\psi_v(s)}(\sigma)
\]
for all $v \in V \setminus V'(s,\sigma)$.
Setting
\[
V':= \bigcup_{\substack{s \in B_r(1_M)\\ \sigma \in \Sigma}} V'(s,\sigma)
\]
we have $|V'| \leq |B_r(1_M)| \cdot |\Sigma| \cdot \varepsilon |V|$ and condition (*) holds for all
$v \in V \setminus V'$.
\par
Let now $s$ and $t$ be distinct elements in $B_r(1_M)$.   
Since $B_r \subset K$ and $\varphi$ is $(K,1 - \varepsilon)$-injective, 
we have that 
$d_V^{\text{Ham}}(\varphi(s), \varphi(t)) \geq 1 - \varepsilon$. This means that we can find a subset
$V''(s,t) \subset V$ of cardinality $|V''(s,t)| \leq \varepsilon |V|$ such that
\begin{equation}
\label{eq:psi-s-t}
\psi_v(s) = v^{\varphi(s)} \neq v^{\varphi(t)} = \psi_v(t)
\end{equation}
for all $v \in V \setminus V''(s,t)$. Setting
\[
V'' :=  \bigcup_{\substack{s,t \in B_r\\s \neq t}} V''(s,t)
\]
we have $|V''| \leq |B_r(1_M)|^2 \cdot \varepsilon |V|$ and it follows from \eqref{eq:psi-s-t} that condition (**) holds for all $v \in V \setminus V''$.

Consequently, conditions (*) and (**) are both satisfied for all $v \in V \setminus (V' \cup V'')$, that is, $V \setminus (V' \cup V'') \subset V_0$.
As
\[
\vert V'  \cup V'' \vert \leq \left(|B_r(1_M)| \cdot |\Sigma| +  |B_r(1_M)|^2\right) \varepsilon |V| = \delta |V|,
\]
where the equality follows from \eqref{eq:delta-espilon},
we finally get
\[
|V(r)| \geq |V_0| \geq (1 - \delta)|V|.
\]
Thus $(M,\Sigma)$ satisfies the Weiss condition.
This proves Theorem \ref{t:weiss-characterization}.(2).
\end{proof}

From  Theorem \ref{t:weiss-characterization} we deduce the following graph theoretic characterization of soficity for finitely generated left-cancellative monoids.
\begin{corollary}
\label{c:weiss-characterization}
Let $M$ be a finitely generated left-cancellative monoid and let $\Sigma \subset M$ be a finite generating subset of $M$. 
Then the following conditions are equivalent:
\begin{enumerate}[{\rm (a)}]
\item 
the monoid $M$ is sofic;
\item 
the pair $(M,\Sigma)$ satisfies the Weiss condition.
\end{enumerate}
As a consequence, for a finitely generated left-cancellative monoid $M$, the fact that the pair
$(M,\Sigma)$ satisfies the Weiss condition or not, is independent of the particular finite generating subset $\Sigma \subset M$.
\end{corollary}

Here follow some examples that illustrate Theorem \ref{t:weiss-characterization}.

We first observe that when $M$ is a finite group, then for any generating  
subset $\Sigma \subset M$ the corresponding Cayley graph $\CC(M,\Sigma) = (V,E)$
satisfies $V(r) = V$ so that condition \eqref{eq:condition-weiss} 
is verified for all $r \in \N$ and $\delta > 0$.
Indeed, in this case, the Cayley graph is vertex-homogeneous since the action of the group
$M$ on $\CC(M,\Sigma)$ induced by right-multiplication is vertex-transitive.
However, this is no more the case, in general, for finite sofic monoids.
When $M$ is a finite monoid, given $r \in \N$ and a generating subset $\Sigma\subset M$,
we can only guarantee that the corresponding Cayley graph $\CC(M,\Sigma) =(V,E)$ satisfies 
$\vert V(r) \vert \geq 1$ (since $1_M \in V(r)$).
This is the reason why  some ``geometric amplification'' of these Cayley graphs is necessary
in order to fulfill condition \eqref{eq:condition-weiss}, as shown in the following two examples. 

\begin{example}
\label{e:sz}
Let $M = \{1_M,a\}$ be the monoid with two elements, where $a^2=a$. 
This is the simplest monoid which is not a group (it is not cancellative).
Taking $\Sigma = \{a\}$ as a generating subset, we have that the Cayley graph $\CC=\CC(M,\Sigma)$ consists of the two vertices $1_M$ and $a$, and the two edges $(1_M,a,a)$ and $(a,a,a)$ (the latter is a loop), see Figure \ref{fig:1}. 
Let now $X$ be a finite set such that $x \neq a$ for all $x \in X$, and consider the $\Sigma$-labeled graph $\GG =(V,E)$ where $V = X \cup \{a\}$ and $E = \{(v,a,a): v \in V\}$, see Figure \ref{fig:1}.
Let $r \in \N$ and $\delta >0$. We have that $B_r^\CC(1_M)$ is isomorphic, as a pointed $\Sigma$-labeled graph, to $B_r^\GG(x)$ for all $x \in X$. This shows that $V(r) \supset X$. Thus, provided $|X|$ is large ($|X| \geq 1/\delta$ would suffice), we have $|V(r)| \geq |X| = |V|-1 \geq (1-\delta)|V|$ and condition \eqref{eq:condition-weiss} is satisfied.
It follows from Theorem \ref{t:weiss-characterization} that $M$ is sofic.

\begin{figure}
\unitlength=0,3mm

\textwidth = 16.00cm \textheight = 22.00cm \oddsidemargin= 0.12in
\evensidemargin = 0.12in \setlength{\parindent}{8pt}
\setlength{\parskip}{5pt plus 2pt minus 1pt}
\setloopdiam{17}\setprofcurve{28}
\begin{center}
\thicklines
\begin{picture}(220,120)
\letvertex A=(0,70)\letvertex B=(60,70)
\drawvertex(A){$\bullet$}\drawvertex(B){$\bullet$}
\put(-21,65){{\bf $1_K$}}
\put(50,60){$a$}

\put(84,67){\footnotesize $a$}
\put(30,73){\footnotesize $a$}

\drawundirectededge(A,B){}
\drawundirectedloop[r](B){}

\letvertex C=(160,70)
\letvertex D=(220,70)
\put(213,55){$a$}
\put(244,68){\footnotesize $a$}

\drawundirectededge(C,D){}
\letvertex CC=(160,110)
\letvertex CCC=(160,90)
\letvertex CCCC=(160,70)
\letvertex CCCCD=(160,50)
\letvertex CCCCC=(160,30)

\drawvertex(C){$\bullet$}
\drawvertex(CC){$\bullet$}
\drawvertex(CCC){$\bullet$}
\drawvertex(CCCC){$\bullet$}
\drawvertex(CCCCD){$\vdots$}
\drawvertex(CCCCC){$\bullet$}

\put(147,108){$x$}
\put(145,88){$x'$}
\put(141,68){$x''$}
\put(140,28){$x'''$}

\put(187,94){\footnotesize $a$}
\put(187,82){\footnotesize $a$}
\put(187,72){\footnotesize $a$}
\put(187,54){\footnotesize $a$}

\drawundirectededge(CC,D){}
\drawundirectededge(CCC,D){}
\drawundirectededge(CCCC,D){}
\drawundirectededge(CCCCC,D){}

\drawundirectedloop[r](D){}
\drawvertex(D){$\bullet$}

\letvertex Z=(187,91)
\letvertex ZZ=(189,90)
\drawedge(ZZ,Z){}

\letvertex X=(187,81)
\letvertex XX=(190,80)
\drawedge(XX,X){}

\letvertex Y=(187,70)
\letvertex YY=(189,70)
\drawedge(YY,Y){}

\letvertex T=(187,49)
\letvertex TT=(189,50)
\drawedge(TT,T){}

\letvertex S=(240,68)
\letvertex SS=(240,69)
\drawedge(S,SS){}

\letvertex U=(80,68)
\letvertex UU=(80,69)
\drawedge(U,UU){}

\letvertex Y=(37,70)
\letvertex YY=(39,70)
\drawedge(YY,Y){}

\put(0,0){$\CC = \CC(M,\Sigma)$}
\put(160,0){$\GG = (V,E)$}
\end{picture}
\end{center}
\caption{The labeled graphs $\CC$ and $\GG$ in Example \ref{e:sz}}
\label{fig:1}
\end{figure}
\end{example}

\begin{figure}
\unitlength=0,3mm
\textwidth = 16.00cm \textheight = 22.00cm \oddsidemargin= 0.12in
\evensidemargin = 0.12in \setlength{\parindent}{8pt}
\setlength{\parskip}{5pt plus 2pt minus 1pt}
\setloopdiam{17}\setprofcurve{10}
\begin{center}

\thicklines
\begin{picture}(520,220)

\letvertex A=(0,70)\drawvertex(A){$\bullet$}\put(-12,55){\footnotesize $1_M$}
\letvertex AA=(80,70)\drawvertex(AA){$\bullet$}\put(75,58){$a$}
\letvertex AAA=(0,150)\drawvertex(AAA){$\bullet$}\put(-12,165){$c_0$}
\letvertex AAAA=(80,150)\drawvertex(AAAA){$\bullet$}\put(75,138){$c_1$}
\drawundirectedloop[l](AAA){}
\drawundirectededge(A,AAA){}
\drawundirectededge(AAA,AAAA){}
\drawundirectededge(AA,AAA){}
\drawundirectedcurvededge(A,AA){}
\drawundirectedcurvededge(AA,A){}
\drawundirectedcurvededge(AAA,AAAA){}
\drawundirectedcurvededge(AAAA,AAA){}
\put(0,0){$\CC = \CC(M,\Sigma)$}

\letvertex AT=(0,110)
\letvertex ATT=(0,115)
\drawedge(ATT,AT){\footnotesize$c_0$}

\letvertex AB=(-20,147)
\letvertex ABB=(-20,148)
\drawedge(AB,ABB){\footnotesize$c_0$}

\letvertex AS=(35,80)
\letvertex ASS=(40,80)
\drawedge(AS,ASS){\footnotesize$a$}

\letvertex AU=(40,60)
\letvertex AUU=(45,60)
\drawedge(AUU,AU){\footnotesize$a$}

\letvertex AW=(41,110)
\letvertex AWW=(40,111)
\drawedge(AWW,AW){\footnotesize$c_0$}

\letvertex AV=(35,160)
\letvertex AVV=(40,160)
\drawedge(AV,AVV){\footnotesize$a$}

\letvertex AZ=(43,140)
\letvertex AZZ=(45,140)
\drawedge(AZZ,AZ){\footnotesize$a$}

\letvertex AX=(40,150)
\letvertex AXX=(39,150)
\drawedge(AXX,AX){\footnotesize$c_0$}

\letvertex C=(130,200)\drawvertex(C){$\bullet$}\put(125,211){\footnotesize $0$}
\letvertex CC=(190,200)\drawvertex(CC){$\bullet$}\put(195,198){\footnotesize $1$}
\drawundirectedloop[l](C){}
\drawundirectedcurvededge(C,CC){}
\drawundirectedcurvededge(CC,C){}
\drawundirectededge(C,CC){}

\put(130,160){$\GG_1 = (V_1,E_1)$}

\letvertex CV=(157,210)
\letvertex CVV=(158,210)
\drawedge(CV,CVV){\footnotesize$a$}

\letvertex CZ=(158,200)
\letvertex CZZ=(157,200)
\drawedge(CZZ,CZ){\footnotesize$c_0$}

\letvertex CX=(160,190)
\letvertex CXX=(161,190)
\drawedge(CXX,CX){\footnotesize$a$}

\letvertex CQ=(110,197)
\letvertex CQQ=(110,198)
\drawedge(CQ,CQQ){\footnotesize$c_0$}

\letvertex CB=(120,88)
\letvertex CBB=(120,87)
\drawedge(CB,CBB){\footnotesize$c_0$}

\letvertex CD=(158,120)
\letvertex CDD=(157,120)
\drawedge(CDD,CD){\footnotesize$c_0$}

\letvertex CW=(158,82)
\letvertex CWW=(157,83)
\drawedge(CWW,CW){\footnotesize$c_0$}

\letvertex DW=(166,90)
\letvertex DWW=(165,91)
\drawedge(DWW,DW){\footnotesize$a$}

\letvertex DDW=(156,70)
\letvertex DDWW=(155,71)
\drawedge(DDW,DDWW){\footnotesize$a$}


\letvertex B=(120,40)\drawvertex(B){$\bullet$}\put(105,25){\footnotesize $(0,1)$}
\letvertex BB=(200,40) \drawvertex(BB){$\bullet$}\put(187,25){\footnotesize $(1,1)$}
\letvertex BBB=(120,120) \drawvertex(BBB){$\bullet$}\put(105,135){\footnotesize $(0,0)$}
\letvertex BBBB=(200,120)  \drawvertex(BBBB){$\bullet$}\put(187,128){\footnotesize $(1,0)$}
\drawundirectedloop[l](BBB){}
\drawundirectededge(B,BBB){}
\drawundirectededge(BB,BBB){}
\drawundirectededge(BBBB,BBB){}
\drawundirectedcurvededge(BBB,BB){}
\drawundirectedcurvededge(BB,BBB){}
\drawundirectedcurvededge(B,BBBB){}
\drawundirectedcurvededge(BBBB,B){}
\put(130,0){$\GG_2 = (V_2,E_2)$}

\letvertex DQ=(100,117)
\letvertex DQQ=(100,118)
\drawedge(DQ,DQQ){\footnotesize$c_0$}


\letvertex DQ=(260,117)
\letvertex DQQ=(260,118)
\drawedge(DQ,DQQ){\footnotesize$c_0$}

\letvertex EQ=(448,143)
\letvertex EQQ=(448,144)
\drawedge(EQ,EQQ){\footnotesize$a$}

\letvertex EEQ=(432,94)
\letvertex EEQQ=(432,93)
\drawedge(EEQ,EEQQ){\footnotesize$a$}

\letvertex FQ=(389,143)
\letvertex FQQ=(389,144)
\drawedge(FQ,FQQ){\footnotesize$a$}

\letvertex FFQ=(372,94)
\letvertex FFQQ=(372,93)
\drawedge(FFQ,FFQQ){\footnotesize$a$}

\letvertex HD=(358,120)
\letvertex HDD=(357,120)

\letvertex HHD=(378,120)
\letvertex HHDD=(377,120)
\drawedge(HHDD,HHD){\footnotesize$c_0$}

\letvertex L=(302,87)
\letvertex LL=(301,90)
\drawedge(LL,L){\footnotesize$c_0$}

\letvertex SL=(297,147)
\letvertex SLL=(296,144)
\drawedge(SLL,SL){\footnotesize$c_0$}

\letvertex TL=(338,86)
\letvertex TLL=(337,87)
\drawedge(TLL,TL){\footnotesize$c_0$}

\letvertex TS=(383,110)
\letvertex TSS=(381,110)
\drawedge(TS,TSS){\footnotesize$a$}

\letvertex TTS=(378,130)
\letvertex TTSS=(377,130)
\drawedge(TTSS,TTS){\footnotesize$a$}

\letvertex GQ=(330,133)
\letvertex GQQ=(330,134)
\drawedge(GQ,GQQ){\footnotesize$a$}

\letvertex GGQ=(311,104)
\letvertex GGQQ=(311,103)
\drawedge(GGQ,GGQQ){\footnotesize$a$}

\letvertex M=(404,74)
\letvertex MM=(401,75)
\drawedge(MM,M){\footnotesize$c_0$}

\letvertex N=(401,165)
\letvertex NN=(404,166)
\drawedge(N,NN){\footnotesize$c_0$}

\letvertex P=(341,156)
\letvertex PP=(344,157)
\drawedge(P,PP){\footnotesize$c_0$}

\letvertex D=(280,120)\drawvertex(D){$\bullet$}\put(240,135){\footnotesize $(0,0,0)$}
\drawundirectedloop[l](D){}
\letvertex DD=(480,120)\drawvertex(DD){$\bullet$}\put(485,115){\footnotesize $(1,1,1)$}
\letvertex E=(320,60)\drawvertex(E){$\bullet$}\put(300,40){\footnotesize $(1,1,0)$}
\letvertex EE=(380,60)\drawvertex(EE){$\bullet$}\put(360,40){\footnotesize $(1,0,1)$}
\letvertex EEE=(440,60)\drawvertex(EEE){$\bullet$}\put(420,40){\footnotesize $(0,1,1)$}
\letvertex F=(320,180)\drawvertex(F){$\bullet$}\put(300,190){\footnotesize $(0,0,1)$}
\letvertex FF=(380,180)\drawvertex(FF){$\bullet$}\put(360,190){\footnotesize $(0,1,0)$}
\letvertex FFF=(440,180)\drawvertex(FFF){$\bullet$}\put(420,190){\footnotesize $(1,0,0)$}
\drawundirectededge(DD,D){}
\drawundirectededge(E,D){}
\drawundirectededge(EE,D){}
\drawundirectededge(EEE,D){}
\drawundirectededge(F,D){}
\drawundirectededge(FF,D){}
\drawundirectededge(FFF,D){}
\drawundirectedcurvededge(D,DD){}
\drawundirectedcurvededge(DD,D){}
\drawundirectedcurvededge(E,F){}
\drawundirectedcurvededge(F,E){}
\drawundirectedcurvededge(EE,FF){}
\drawundirectedcurvededge(FF,EE){}
\drawundirectedcurvededge(EEE,FFF){}
\drawundirectedcurvededge(FFF,EEE){}
\put(320,0){$\GG_3 = (V_3,E_3)$}
\end{picture}
\end{center}
\caption{The labeled graphs $\CC$, $\GG_1$, $\GG_2$ and $\GG_3$ in Example \ref{e:map-2}}
\label{fig:4}
\end{figure} 
\begin{example}
\label{e:map-2}
Let $M = \Map(\{0,1\})$ be the symmetric monoid of the set $X=\{0,1\}$. Then $M = \{1_M,a,c_0,c_1\}$
where $a(x) = \bar{x}:=1-x$ and $c_y(x) = y$ for all $x,y \in X$. In other words, $a$ is the bijective map which exchanges $0$ and $1$ and $c_y$ is the constant map with value $y$. Taking $\Sigma = \{a,c_0\}$ as a generating subset of $M$, we have that the Cayley graph $\CC=\CC(M,\Sigma)$ consists of the four vertices $m \in M$ and the edges $(1_M,a,a)$, $(a,a,1_M)$, $(c_x,a,c_{\bar{x}})$, for $x \in X$, and $(m,c_0,c_0)$, for $m \in M$, see Figure \ref{fig:4}.

Let $r \in \N$ and $\delta >0$. Let $n \in \N$ and consider now the $\Sigma$-labeled graph $\GG_n = (V_n,E_n)$, where $V_n := X^n$ and $E_n = \{(x,\sigma,x^\sigma): x \in X^n, \sigma \in \Sigma\}$, where for $x = (x_1,x_2, \ldots, x_n) \in X$ and $\sigma \in \Sigma$ we set $x^\sigma = (\sigma(x_1), \sigma(x_2), \ldots, \sigma(x_n)) \in X^n$. In other words, $\GG_n$ is the \emph{Schreier graph} for the diagonal action of $M$ on $X^n$ with respect to $\Sigma$.
In Figure \ref{fig:4} we have drawn $\GG_1$, $\GG_2$, and $\GG_3$.
It is clear that $V_n(r) \supset X^n \setminus \Delta_n$, where $\Delta_n:=\{(x,x,\ldots,x): x \in X\}$ is the diagonal of $X^n$. It follows that, provided $n \geq 1-\log_2(\delta)$, one has $|V_n(r)| \geq |X^n| - |\Delta_n| = 2^n-2 \geq (1 - \delta) |V_n|$, so that condition \eqref{eq:condition-weiss} is satisfied by $\GG_n$.
\end{example}

\begin{example}
\label{e:naturals}
Let $M = \N = \{0,1,2,\ldots\}$ be the additive monoid of the naturals. Taking $\Sigma = \{1\}$
as a generating subset, we have that the Cayley graph $\CC=\CC(\N,\Sigma)$ consists of the vertices $0,1,\ldots$ and the edges $(n,1,n+1)$ for all $n \in \N$, see Figure \ref{fig:2}.
Let $r \in \N$ and $\delta >0$. Let $X$ be a finite set disjoint from $\N$ and consider the $\Sigma$-labeled graph $\GG =(V,E)$ where $V = X \cup \{1,2,\ldots,r\}$ and $E = \{(x,1,1): x \in X\} \cup \{(k,1,k+1):  1 \leq k \leq r-1\}$, see Figure \ref{fig:2}. As in Example \ref{e:sz},
we have $V(r) \supset X$ so that, provided $|X|$ is large ($|X| \geq r/\delta$ would suffice)
condition \eqref{eq:condition-weiss} is satisfied by $\GG$.

\begin{figure}
\unitlength=0,3mm
\textwidth = 16.00cm \textheight = 22.00cm \oddsidemargin= 0.12in
\evensidemargin = 0.12in \setlength{\parindent}{8pt}
\setlength{\parskip}{5pt plus 2pt minus 1pt}
\setloopdiam{17}\setprofcurve{28}
\begin{center}

\thicklines
\begin{picture}(500,120)
\letvertex Z=(-40,70)
\letvertex A=(0,70)\letvertex B=(40,70) \letvertex BB=(80,70) \letvertex BBB=(120,70)  
\letvertex BBBB=(160,70)
\put(-2,55){\footnotesize $0$}
\put(38,55){\footnotesize $1$}
\put(78,55){\footnotesize $2$}
\put(118,58){\footnotesize $n$}
\put(148,58){\footnotesize $n\!\!+\!\!1$}

\drawvertex(A){$\bullet$}\drawvertex(B){$\bullet$}\drawvertex(BB){$\bullet$}
\drawvertex(BBB){$\bullet$} \drawvertex(BBBB){$\bullet$}

\put(170,67){$\cdots$}
\put(90,67){$\cdots$}

\drawundirectededge(A,BB){}
\drawundirectededge(BBB,BBBB){}

\letvertex I=(24,70)
\letvertex II=(26,70)
\drawedge(II,I){}

\letvertex L=(64,70)
\letvertex LL=(66,70)
\drawedge(LL,L){}

\letvertex M=(144,70)
\letvertex MM=(146,70)
\drawedge(MM,M){}

\letvertex C=(260,70)
\letvertex D=(320,70)
\letvertex E=(360,70)
\letvertex F=(400,70)
\letvertex G=(440,70)
\letvertex H=(480,70)
\put(412,67){$\cdots$}

\drawundirectededge(D,F){}
\drawundirectededge(G,H){}
\drawundirectededge(C,D){}
\letvertex CC=(260,110)
\letvertex CCC=(260,90)
\letvertex CCCC=(260,70)
\letvertex CCCCD=(260,50)
\letvertex CCCCC=(260,30)

\drawvertex(C){$\bullet$}
\drawvertex(CC){$\bullet$}
\drawvertex(CCC){$\bullet$}
\drawvertex(CCCC){$\bullet$}
\drawvertex(CCCCD){$\vdots$}
\drawvertex(CCCCC){$\bullet$}

\drawvertex(D){$\bullet$}
\drawvertex(E){$\bullet$}
\drawvertex(F){$\bullet$}
\drawvertex(G){$\bullet$}
\drawvertex(H){$\bullet$}

\drawundirectededge(D,F){}
\drawundirectededge(G,H){}

\put(247,108){$x$}
\put(245,88){$x'$}
\put(241,68){$x''$}
\put(240,28){$x'''$}

\put(318,55){\footnotesize $1$}
\put(358,55){\footnotesize $2$}
\put(398,55){\footnotesize $3$}
\put(428,58){\footnotesize $r\!\!-\!\!1$}
\put(478,58){\footnotesize $r$}

\put(410,67){$\cdots$}

\drawundirectededge(CC,D){}
\drawundirectededge(CCC,D){}
\drawundirectededge(CCCC,D){}
\drawundirectededge(CCCCC,D){}

\letvertex Z=(287,91)
\letvertex ZZ=(289,90)
\drawedge(ZZ,Z){}

\letvertex X=(287,81)
\letvertex XX=(290,80)
\drawedge(XX,X){}

\letvertex Y=(287,70)
\letvertex YY=(289,70)
\drawedge(YY,Y){}

\letvertex T=(287,49)
\letvertex TT=(289,50)
\drawedge(TT,T){}

\letvertex S=(344,70)
\letvertex SS=(346,70)
\drawedge(SS,S){}

\letvertex U=(384,70)
\letvertex UU=(386,70)
\drawedge(UU,U){}

\letvertex V=(464,70)
\letvertex VV=(466,70)
\drawedge(VV,V){}
\put(60,0){$\CC = \CC(\N,\Sigma)$}
\put(330,0){$\GG = (V,E)$}
\end{picture}
\end{center}
\caption{The labeled graphs $\CC$ and $\GG$ in Example \ref{e:naturals}}
\label{fig:2}
\end{figure}

Alternatively, let $n \in \N$. Consider first the $\Sigma$-labeled graph $\GG_n = (V_n,E_n)$, where $V_n = \{0,1,2,\ldots,n-1\}$ and $E_n = \{(k,1,k+1): k=0,1,\ldots,n-2\} \cup \{(n-1,1,0)\}$, see Figure \ref{fig:3}. Now if $n \geq r+1$, one has $V_n(r) = V_n$ so that condition \eqref{eq:condition-weiss} is satisfied by $\GG_n$ (note that this is, in fact, independent of
$\delta$).

\begin{figure}
\unitlength=0,3mm
\textwidth = 16.00cm \textheight = 22.00cm \oddsidemargin= 0.12in
\evensidemargin = 0.12in \setlength{\parindent}{8pt}
\setlength{\parskip}{5pt plus 2pt minus 1pt}
\setloopdiam{17}\setprofcurve{28}
\begin{center}
\thicklines

\begin{picture}(500,120)
\letvertex A=(0,40) \drawvertex(A){$\bullet$} \put(-2,25){\footnotesize $0$}
\letvertex AA=(40,40) \drawvertex(AA){$\bullet$} \put(38,25){\footnotesize $1$}
\letvertex AAA=(80,40) \drawvertex(AAA){$\bullet$} \put(78,25){\footnotesize $2$}
\letvertex AAAA=(120,40) \drawvertex(AAAA){$\bullet$} \put(108,28){\footnotesize $r\!\!-\!\!1$}

\letvertex B=(0,80) \drawvertex(B){$\bullet$} \put(-10,90){\footnotesize $n\!\!-\!\!1$}
\letvertex BB=(40,80) \drawvertex(BB){$\bullet$} \put(30,90){\footnotesize $n\!\!-\!\!2$}
\letvertex BBB=(80,80) \drawvertex(BBB){$\bullet$} \put(70,90){\footnotesize $r\!\!+\!\!1$}
\letvertex BBBB=(120,80) \drawvertex(BBBB){$\bullet$} \put(120,90){\footnotesize $r$}

\put(200,50){$\overset{V(r)}{\overbrace{\ \hspace{4.5cm} \ }}$}
\drawundirectededge(A,AAA){}
\drawundirectededge(A,B){}
\drawundirectededge(B,BB){}
\drawundirectededge(BBB,BBBB){}
\drawundirectededge(AAAA,BBBB){}

\letvertex K=(18,80)
\letvertex KK=(16,80)
\drawedge(KK,K){}

\letvertex X=(64,40)
\letvertex XX=(66,40)
\drawedge(XX,X){}

\letvertex T=(36,80)
\letvertex TT=(34,80)
\drawedge(TT,T){}

\letvertex S=(96,80)
\letvertex SS=(98,80)
\drawedge(S,SS){}

\letvertex Y=(24,40)
\letvertex YY=(26,40)
\drawedge(YY,Y){}

\letvertex U=(0,56)
\letvertex UU=(0,58)
\drawedge(U,UU){}

\letvertex V=(120,62)
\letvertex VV=(120,64)
\drawedge(VV,V){}
\put(50,77){$\cdots$}
\put(90,37){$\cdots$}
\put(290,37){$\cdots$}
\put(410,37){$\cdots$}

\letvertex C=(200,40)\drawvertex(C){$\bullet$}\put(198,25){\footnotesize $0$}
\letvertex D=(240,40)\drawvertex(D){$\bullet$}\put(238,25){\footnotesize $1$}
\letvertex E=(280,40)\drawvertex(E){$\bullet$}\put(278,25){\footnotesize $2$}
\letvertex F=(320,40)\drawvertex(F){$\bullet$}\put(300,25){\footnotesize $n\!\!-\!\!r\!\!-\!\!2$}
\letvertex G=(360,40)\drawvertex(G){$\bullet$}\put(350,25){\footnotesize $n\!\!-\!\!r\!\!-\!\!1$}
\letvertex H=(400,40)\drawvertex(H){$\bullet$}\put(398,25){\footnotesize $n\!\!-\!\!r$}
\letvertex I=(440,40)\drawvertex(I){$\bullet$}\put(435,25){\footnotesize $n\!\!-\!\!2$}
\letvertex L=(480,40)\drawvertex(L){$\bullet$}\put(475,25){\footnotesize $n\!\!-\!\!1$}

\drawundirectededge(C,E){}
\drawundirectededge(F,H){}
\drawundirectededge(I,L){}

\letvertex M=(224,40)
\letvertex MM=(226,40)
\drawedge(MM,M){}

\letvertex N=(264,40)
\letvertex NN=(266,40)
\drawedge(NN,N){}

\letvertex O=(344,40)
\letvertex OO=(346,40)
\drawedge(OO,O){}

\letvertex P=(384,40)
\letvertex PP=(386,40)
\drawedge(PP,P){}

\letvertex Q=(464,40)
\letvertex QQ=(466,40)
\drawedge(QQ,Q){}

\put(20,0){$\GG_n = (V_n,E_n)$}
\put(260,0){$\GG_n' = (V_n',E_n')$}
\end{picture}
\end{center}
\caption{The labeled graphs $\GG_n$ and $\GG_n'$ in Example \ref{e:naturals}}
\label{fig:3}
\end{figure}

Finally, consider the $\Sigma$-labeled graph $\GG_n' = (V_n',E_n')$, where $V_n' = \{0,1,2,\ldots,n-1\}$ and $E_n' = \{(k,1,k+1): k=0,1,\ldots,n-2\}$, see Figure \ref{fig:3}. 
Now if $n \geq r/\delta$, one has $V_n(r) = \{0,1,\ldots, n-r-1\}$
and $|V_n(r)| = n-r \geq (1 - \delta)n = (1 - \delta)|V_n|$, so that condition \eqref{eq:condition-weiss} is satisfied also by $\GG_n'$. 
\end{example}

\begin{example}
Let $M$ be a cancellative, right-amenable, finitely generated monoid. 
Let $\Sigma \subset M$ be a finite generating subset.
It follows from \cite[Proposition 3.2]{semieden} and/or \cite[Proposition 2.3]{fekete}
(after passing to the opposite monoid) that given any finite subset $K \subset M$ and $\delta > 0$ one can find a finite subset $\Omega = \Omega(K,\delta) \subset M$ such that 
\begin{equation}
\label{e:interior}
|\Int_K^*(\Omega)| \geq (1 - \delta)|\Omega|,
\end{equation} 
where $\Int_K^*(\Omega) = \{s \in M: sK \subset \Omega\}$. 
Fix $r \in \N^+$ and $\delta > 0$. Set $K:= B_r(1_M)$ and denote by $\GG = (V,E)$ the subgraph
induced by $V := \Omega(K,\delta)$ in $\CC$. Then $sK = B_r(s)$ so that $V(r) \supset \Int_K^*(\Omega)$. From \eqref{e:interior} we deduce that the finite $\Sigma$-labeled graph $\GG$ satisfies \eqref{eq:condition-weiss}.
\end{example}

\begin{example}
Recall that the bicyclic monoid $B$ is not sofic by Theorem~\ref{t:bicyclic-not-sofic} (and not left-cancellative either).
We show that for $\Sigma = \{p,q\}$ the pair $(B,\Sigma)$ does not satisfy the Weiss condition 
by looking at the ball $B_r(1_B)$ (see Figure~\ref{fig:BallB}) in the Cayley graph of $B$ (see Figure~\ref{fig:CayB}). 
\begin{figure}
\unitlength=0,3mm

\textwidth = 16.00cm \textheight = 22.00cm \oddsidemargin= 0.12in
\evensidemargin = 0.12in \setlength{\parindent}{8pt}
\setlength{\parskip}{5pt plus 2pt minus 1pt}
\setloopdiam{17}\setprofcurve{18}

\begin{center}
\thinlines
\begin{picture}(420,310)
\letvertex A=(10,10)\letvertex AA=(60,10) \letvertex AAA=(110,10)
\letvertex AAAA=(160,10)\letvertex AAAAA=(260,10) \letvertex AAAAAA=(310,10)
\drawvertex(A){$\bullet$}\drawvertex(AA){$\bullet$}
\drawvertex(AAA){$\bullet$}\drawvertex(AAAA){$\bullet$}
\drawvertex(AAAAA){$\bullet$}\drawvertex(AAAAAA){$\bullet$}

\letvertex AZ=(200,10)\letvertex AZZ=(210,10) \letvertex AZZZ=(220,10)
\drawvertex(AZ){$\bullet$}\drawvertex(AZZ){$\bullet$}\drawvertex(AZZZ){$\bullet$}

\letvertex ABZ=(350,10)\letvertex ABZZ=(360,10) \letvertex ABZZZ=(370,10)
\drawvertex(ABZ){$\bullet$}\drawvertex(ABZZ){$\bullet$}\drawvertex(ABZZZ){$\bullet$}

\letvertex BZ=(200,60)\letvertex BZZ=(210,60) \letvertex BZZZ=(220,60)
\drawvertex(BZ){$\bullet$}\drawvertex(BZZ){$\bullet$}\drawvertex(BZZZ){$\bullet$}
\letvertex BBZ=(350,60)\letvertex BBZZ=(360,60) \letvertex BBZZZ=(370,60)
\drawvertex(BBZ){$\bullet$}\drawvertex(BBZZ){$\bullet$}\drawvertex(BBZZZ){$\bullet$}

\letvertex CZ=(200,110)\letvertex CZZ=(210,110) \letvertex CZZZ=(220,110)
\drawvertex(CZ){$\bullet$}\drawvertex(CZZ){$\bullet$}\drawvertex(CZZZ){$\bullet$}
\letvertex ACZ=(350,110)\letvertex ACZZ=(360,110) \letvertex ACZZZ=(370,110)
\drawvertex(ACZ){$\bullet$}\drawvertex(ACZZ){$\bullet$}\drawvertex(ACZZZ){$\bullet$}

\letvertex DZ=(200,210)\letvertex DZZ=(210,210) \letvertex DZZZ=(220,210)
\drawvertex(DZ){$\bullet$}\drawvertex(DZZ){$\bullet$}\drawvertex(DZZZ){$\bullet$}
\letvertex ADZ=(350,210)\letvertex ADZZ=(360,210) \letvertex ADZZZ=(370,210)
\drawvertex(ADZ){$\bullet$}\drawvertex(ADZZ){$\bullet$}\drawvertex(ADZZZ){$\bullet$}

\letvertex Q=(10,150) \letvertex QQ=(10,160) \letvertex QQQ=(10,170) 
\drawvertex(Q){$\bullet$} \drawvertex(QQ){$\bullet$} \drawvertex(QQQ){$\bullet$}

\letvertex B=(10,60)\letvertex BB=(60,60) \letvertex BBB=(110,60)
\letvertex BBBB=(160,60)\letvertex BBBBB=(260,60) \letvertex BBBBBB=(310,60)
\drawvertex(B){$\bullet$}\drawvertex(BB){$\bullet$}
\drawvertex(BBB){$\bullet$}\drawvertex(BBBB){$\bullet$}
\drawvertex(BBBBB){$\bullet$}\drawvertex(BBBBBB){$\bullet$}

\letvertex C=(10,110)\letvertex CC=(60,110) \letvertex CCC=(110,110)
\letvertex CCCC=(160,110)\letvertex CCCCC=(260,110) \letvertex CCCCCC=(310,110)
\drawvertex(C){$\bullet$}\drawvertex(CC){$\bullet$}
\drawvertex(CCC){$\bullet$}\drawvertex(CCCC){$\bullet$}
\drawvertex(CCCCC){$\bullet$}\drawvertex(CCCCCC){$\bullet$}

\letvertex D=(10,210)\letvertex DD=(60,210) \letvertex DDD=(110,210)
\letvertex DDDD=(160,210)\letvertex DDDDD=(260,210) \letvertex DDDDDD=(310,210)
\drawvertex(D){$\bullet$}\drawvertex(DD){$\bullet$}
\drawvertex(DDD){$\bullet$}\drawvertex(DDDD){$\bullet$}
\drawvertex(DDDDD){$\bullet$}\drawvertex(DDDDDD){$\bullet$}

\letvertex CAC=(10,250)
\letvertex CACC=(10,260)
\letvertex CACCC=(10,270)

\drawvertex(CAC){$\bullet$}
\drawvertex(CACC){$\bullet$}
\drawvertex(CACCC){$\bullet$}

\letvertex S=(10,37)\letvertex SA=(10,39)\drawedge(SA,S){}
\letvertex SS=(10,87)\letvertex SSA=(10,89)\drawedge(SSA,SS){}

\put(0,35){\footnotesize $q$} \put(0,85){\footnotesize $q$}
\put(-5,55){$q$} 
\put(-5,105){$q^2$} \put(-9,205){$q^{m}$}

\put(-10,-5){$1_B$}
\put(60,-5){$p$} \put(110,-5){$p^2$} \put(160,-5){$p^3$} \put(260,-5){$p^n$} \put(310,-5){$p^{n+1}$}
\put(55,45){$qp$} \put(105,45){$qp^2$} \put(155,45){$qp^3$} \put(255,45){$qp^n$} \put(305,45){$qp^{n+1}$}
\put(55,95){$q^2p$} \put(105,95){$q^2p^2$} \put(155,95){$q^2p^3$} \put(255,95){$q^2p^n$} \put(305,95){$q^2p^{n+1}$}
\put(55,195){$q^mp$} \put(105,195){$q^mp^2$} \put(155,195){$q^mp^3$} \put(255,195){$q^mp^n$} \put(305,195){$q^mp^{n+1}$}

\letvertex P=(32,28) \letvertex PP=(33,28) \drawedge(P,PP){}
\letvertex AP=(82,28) \letvertex APP=(83,28) \drawedge(AP,APP){}
\letvertex BP=(132,28) \letvertex BPP=(133,28) \drawedge(BP,BPP){}
\letvertex CP=(282,28) \letvertex CPP=(283,28) \drawedge(CP,CPP){}

\letvertex PA=(32,78) \letvertex PPA=(33,78) \drawedge(PA,PPA){}
\letvertex APA=(82,78) \letvertex APPA=(83,78) \drawedge(APA,APPA){}
\letvertex BPA=(132,78) \letvertex BPPA=(133,78) \drawedge(BPA,BPPA){}
\letvertex CPA=(282,78) \letvertex CPPA=(283,78) \drawedge(CPA,CPPA){}

\letvertex PB=(32,128) \letvertex PPB=(33,128) \drawedge(PB,PPB){}
\letvertex APB=(82,128) \letvertex APPB=(83,128) \drawedge(APB,APPB){}
\letvertex BPB=(132,128) \letvertex BPPB=(133,128) \drawedge(BPB,BPPB){}
\letvertex CPB=(282,128) \letvertex CPPB=(283,128) \drawedge(CPB,CPPB){}

\letvertex PC=(32,228) \letvertex PPC=(33,228) \drawedge(PC,PPC){}
\letvertex APC=(82,228) \letvertex APPC=(83,228) \drawedge(APC,APPC){}
\letvertex BPC=(132,228) \letvertex BPPC=(133,228) \drawedge(BPC,BPPC){}
\letvertex CPC=(282,228) \letvertex CPPC=(283,228) \drawedge(CPC,CPPC){}

\put(280,85){\footnotesize $q$}\put(280,67){\footnotesize $p$}

\put(30,35){\footnotesize $q$}\put(80,35){\footnotesize $q$}\put(130,35){\footnotesize $q$}
\put(280,35){\footnotesize $q$}

\put(30,17){\footnotesize $p$}\put(80,17){\footnotesize $p$}\put(130,17){\footnotesize $p$}
\put(280,17){\footnotesize $p$}

\put(30,85){\footnotesize $q$}\put(80,85){\footnotesize $q$}\put(130,85){\footnotesize $q$}
\put(280,35){\footnotesize $q$}

\put(30,67){\footnotesize $p$}\put(80,67){\footnotesize $p$}\put(130,67){\footnotesize $p$}
\put(280,17){\footnotesize $p$}

\put(30,135){\footnotesize $q$}\put(80,135){\footnotesize $q$}\put(130,135){\footnotesize $q$}
\put(280,135){\footnotesize $q$}

\put(30,117){\footnotesize $p$}\put(80,117){\footnotesize $p$}\put(130,117){\footnotesize $p$}
\put(280,117){\footnotesize $p$}

\put(30,235){\footnotesize $q$}\put(80,235){\footnotesize $q$}\put(130,235){\footnotesize $q$}
\put(280,235){\footnotesize $q$}

\put(30,217){\footnotesize $p$}\put(80,217){\footnotesize $p$}\put(130,217){\footnotesize $p$}
\put(280,217){\footnotesize $p$}

\drawundirectededge(A,AAAA){}
\drawundirectededge(B,BBBB){}
\drawundirectededge(C,CCCC){}
\drawundirectededge(D,DDDD){}
\drawundirectededge(A,C){}

\drawundirectededge(AAAAA,AAAAAA){}
\drawundirectededge(BBBBB,BBBBBB){}
\drawundirectededge(CCCCC,CCCCCC){}
\drawundirectededge(DDDDD,DDDDDD){}

\letvertex AB=(40,10)\letvertex AAB=(90,10) \letvertex AAAB=(140,10) \letvertex AAAAAB=(295,10)
\drawedge(A,AB){} \drawedge(AA,AAB){}\drawedge(AAA,AAAB){} \drawedge(AAAAA,AAAAAB){}

\letvertex ZAB=(40,60)\letvertex ZAAB=(90,60) \letvertex ZAAAB=(140,60) \letvertex ZAAAAAB=(295,60)
\drawedge(B,ZAB){} \drawedge(BB,ZAAB){}\drawedge(BBB,ZAAAB){} \drawedge(BBBBB,ZAAAAAB){}

\letvertex CAB=(40,110)\letvertex CAAB=(90,110) \letvertex CAAAB=(140,110) 
\letvertex CAAAAAB=(295,110)
\drawedge(C,CAB){} \drawedge(CC,CAAB){}\drawedge(CCC,CAAAB){} \drawedge(CCCCC,CAAAAAB){}

\letvertex DAB=(40,210)\letvertex DAAB=(90,210) \letvertex DAAAB=(140,210) 
\letvertex DAAAAAB=(295,210)
\drawedge(D,DAB){} \drawedge(DD,DAAB){}\drawedge(DDD,DAAAB){} \drawedge(DDDDD,DAAAAAB){}

\letvertex L=(400,10)
\letvertex LL=(400,60)
\letvertex LLL=(400,110)
\letvertex LLLL=(400,210)

\drawundirectedcurvededge(A,AA){}
\drawundirectedcurvededge(AA,AAA){}
\drawundirectedcurvededge(AAA,AAAA){}
\drawundirectedcurvededge(AAAAA,AAAAAA){}
\put(30,17){\footnotesize $p$}\put(80,17){\footnotesize $p$}\put(130,17){\footnotesize $p$}
\put(280,17){\footnotesize $p$}

\drawundirectedcurvededge(B,BB){}
\drawundirectedcurvededge(BB,BBB){}
\drawundirectedcurvededge(BBB,BBBB){}
\drawundirectedcurvededge(BBBBB,BBBBBB){}

\drawundirectedcurvededge(C,CC){}
\drawundirectedcurvededge(CC,CCC){}
\drawundirectedcurvededge(CCC,CCCC){}
\drawundirectedcurvededge(CCCCC,CCCCCC){}

\drawundirectedcurvededge(D,DD){}
\drawundirectedcurvededge(DD,DDD){}
\drawundirectedcurvededge(DDD,DDDD){}
\drawundirectedcurvededge(DDDDD,DDDDDD){}

\thinlines
\letvertex F=(10,300)
\drawedge(A,F){} \thinlines
\drawedge(A,L){} \drawedge(B,LL){} 
\thinlines
\drawedge(C,LLL){}
\thinlines
\drawedge(D,LLLL){}
\end{picture}
\end{center}
\caption{The Cayley graph $\CC(B,\{p,q\})$ of the bicyclic monoid $B$}
\label{fig:CayB}
\end{figure}
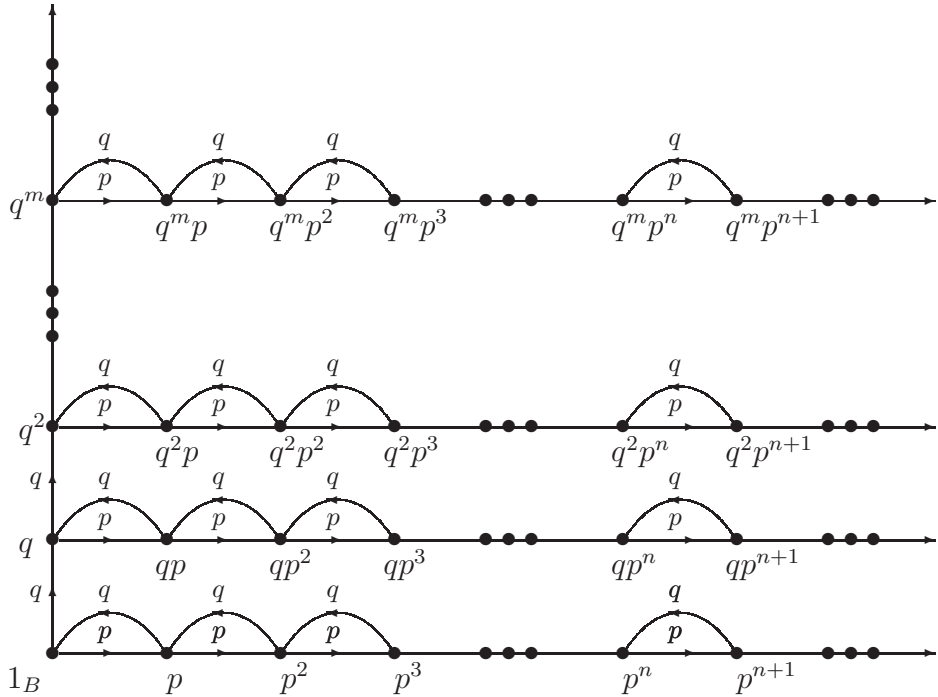
Indeed, if $\GG=(V,E)$ is a finite $\Sigma$-labelled graph and $r \geq 2$,
we observe that the edge labelled $p$ starting at a vertex $v \in V(r)$
arrives at a vertex $v' \notin V(r)$. 
\begin{figure}
\unitlength=0,3mm

\textwidth = 16.00cm \textheight = 22.00cm \oddsidemargin= 0.12in
\evensidemargin = 0.12in \setlength{\parindent}{8pt}
\setlength{\parskip}{5pt plus 2pt minus 1pt}
\setloopdiam{17}\setprofcurve{18}

\begin{center}
\thinlines
\begin{picture}(420,360)
\letvertex A=(10,10)\letvertex AA=(60,10) \letvertex AAA=(110,10)
\letvertex AAAA=(160,10)\letvertex AAAAA=(260,10) \letvertex AAAAAA=(310,10)
\letvertex AAAAVA=(360,10) \letvertex AAAAAVA=(410,10)
\drawvertex(A){$\bullet$}\drawvertex(AA){$\bullet$}
\drawvertex(AAA){$\bullet$}\drawvertex(AAAA){$\bullet$}
\drawvertex(AAAAA){$\bullet$}\drawvertex(AAAAAA){$\bullet$}
\drawvertex(AAAAVA){$\bullet$}\drawvertex(AAAAAVA){$\bullet$}
\drawundirectedcurvededge(AAAAVA,AAAAAVA){}
\drawundirectedcurvededge(AAAAAA,AAAAVA){}
\letvertex AZ=(200,10)\letvertex AZZ=(210,10) \letvertex AZZZ=(220,10)
\drawvertex(AZ){$\bullet$}\drawvertex(AZZ){$\bullet$}\drawvertex(AZZZ){$\bullet$}


\letvertex BZ=(200,60)\letvertex BZZ=(210,60) \letvertex BZZZ=(220,60)
\drawvertex(BZ){$\bullet$}\drawvertex(BZZ){$\bullet$}\drawvertex(BZZZ){$\bullet$}
\letvertex BBZ=(350,60)\letvertex BBZZ=(360,60) \letvertex BBZZZ=(370,60)

\letvertex CZ=(200,110)\letvertex CZZ=(210,110) \letvertex CZZZ=(220,110)
\drawvertex(CZ){$\bullet$}\drawvertex(CZZ){$\bullet$}\drawvertex(CZZZ){$\bullet$}
\letvertex ACZ=(350,110)\letvertex ACZZ=(360,110) \letvertex ACZZZ=(370,110)

\letvertex DZ=(200,210)\letvertex DZZ=(210,210) \letvertex DZZZ=(220,210)
\letvertex ADZ=(350,210)\letvertex ADZZ=(360,210) \letvertex ADZZZ=(370,210)

\letvertex Q=(10,150) \letvertex QQ=(10,160) \letvertex QQQ=(10,170) 
\drawvertex(Q){$\bullet$} \drawvertex(QQ){$\bullet$} \drawvertex(QQQ){$\bullet$}

\letvertex B=(10,60)\letvertex BB=(60,60) \letvertex BBB=(110,60)
\letvertex BBBB=(160,60)\letvertex BBBBB=(260,60) \letvertex BBBBBB=(310,60)
\drawvertex(B){$\bullet$}\drawvertex(BB){$\bullet$}
\drawvertex(BBB){$\bullet$}\drawvertex(BBBB){$\bullet$}
\drawvertex(BBBBB){$\bullet$}\drawvertex(BBBBBB){$\bullet$}

\letvertex C=(10,110)\letvertex CC=(60,110) \letvertex CCC=(110,110)
\letvertex CCCC=(160,110)\letvertex CCCCC=(260,110) \letvertex CCCCCC=(310,110)
\drawvertex(C){$\bullet$}\drawvertex(CC){$\bullet$}
\drawvertex(CCC){$\bullet$}\drawvertex(CCCC){$\bullet$}
\drawvertex(CCCCC){$\bullet$}\drawvertex(CCCCCC){$\bullet$}

\letvertex D=(10,210)\letvertex DD=(60,210) \letvertex DDD=(110,210)
\letvertex DDDD=(160,210)\letvertex DDDDD=(260,210) \letvertex DDDDDD=(310,210)
\drawvertex(D){$\bullet$}\drawvertex(DD){$\bullet$}
\drawvertex(DDD){$\bullet$}\drawvertex(DDDD){$\bullet$}

\letvertex DA=(10,260)\letvertex DDA=(60,260) \letvertex DDDA=(110,260)
\drawvertex(DA){$\bullet$}\drawvertex(DDA){$\bullet$}\drawvertex(DDDA){$\bullet$}

\letvertex DAA=(10,310)\letvertex DDAA=(60,310) 
\drawvertex(DAA){$\bullet$}\drawvertex(DDAA){$\bullet$}

\letvertex DADA=(10,360)\drawvertex(DADA){$\bullet$}

\letvertex CAC=(10,250)
\letvertex CACC=(10,260)
\letvertex CACCC=(10,270)

\letvertex ZDAA=(10,310)\letvertex ZDADA=(60,310) 
\drawundirectededge(ZDAA,ZDADA){}

\letvertex ZZDAA=(10,260)\letvertex ZZDADA=(110,260)\letvertex ZZDDAA=(60,260) 
\drawundirectededge(ZZDAA,ZZDADA){}
\drawundirectedcurvededge(ZDAA,ZDADA){}
\drawundirectedcurvededge(ZZDAA,ZZDDAA){}
\drawundirectedcurvededge(ZZDDAA,ZZDADA){}

\letvertex DZAA=(360,60)\letvertex DZDAA=(310,60) 
\drawvertex(DZAA){$\bullet$}\drawvertex(DZDAA){$\bullet$}
\drawundirectededge(DZDAA,DZAA){}
\drawundirectedcurvededge(DZDAA,DZAA){}
\put(330,35){\footnotesize $q$}\put(330,85){\footnotesize $q$}\put(280,85){\footnotesize $q$}
\put(380,35){\footnotesize $q$}
\put(330,17){\footnotesize $p$}\put(330,67){\footnotesize $p$}\put(280,67){\footnotesize $p$}
\put(380,17){\footnotesize $p$}

\letvertex DZADA=(10,60)\letvertex DZDAA=(310,60) 
\drawundirectededge(DZADA,DZDAA){}
\letvertex ZDZADA=(10,110)\letvertex ZDZDAA=(260,110) 
\drawundirectededge(ZDZADA,ZDZDAA){}

\letvertex ZP=(332,28) \letvertex ZPP=(333,28) \drawedge(ZP,ZPP){}
\letvertex ZAP=(332,78) \letvertex ZAPP=(333,78) \drawedge(ZAP,ZAPP){}
\letvertex ZZP=(382,28) \letvertex ZZPP=(383,28) \drawedge(ZZP,ZZPP){}

\letvertex ZBP=(32,328) \letvertex ZBPP=(33,328) \drawedge(ZBP,ZBPP){}
\letvertex ZCP=(32,278) \letvertex ZCPP=(33,278) \drawedge(ZCP,ZCPP){}
\letvertex ZDP=(82,278) \letvertex ZDPP=(83,278) \drawedge(ZDP,ZDPP){}
\put(30,335){\footnotesize $q$} \put(30,317){\footnotesize $p$}
\put(30,285){\footnotesize $q$} \put(30,267){\footnotesize $p$}
\put(80,285){\footnotesize $q$} \put(80,267){\footnotesize $p$}

\letvertex ZBPZ=(387,10) \letvertex ZBPPZ=(389,10) \drawedge(ZBPPZ,ZBPZ){}
\letvertex ZBPZZ=(337,10) \letvertex ZBPPZZ=(339,10) \drawedge(ZBPPZZ,ZBPZZ){}
\letvertex ZBPZZZ=(337,60) \letvertex ZBPPZZZ=(339,60) \drawedge(ZBPPZZZ,ZBPZZZ){}

\letvertex ZBPZS=(37,310) \letvertex ZBPPZS=(39,310) \drawedge(ZBPPZS,ZBPZS){}
\letvertex ZBPZT=(37,260) \letvertex ZBPPZT=(39,260) \drawedge(ZBPPZT,ZBPZT){}
\letvertex ZBPZU=(87,260) \letvertex ZBPPZU=(89,260) \drawedge(ZBPPZU,ZBPZU){}

\letvertex S=(10,37)\letvertex SA=(10,39)\drawedge(SA,S){}
\letvertex SS=(10,87)\letvertex SSA=(10,89)\drawedge(SSA,SS){}
\letvertex SX=(10,237)\letvertex SXA=(10,239)\drawedge(SXA,SX){}
\letvertex SV=(10,287)\letvertex SVA=(10,289)\drawedge(SVA,SV){}
\letvertex SC=(10,337)\letvertex SCA=(10,339)\drawedge(SCA,SC){}

\put(0,35){\footnotesize $q$} \put(0,85){\footnotesize $q$}
\put(0,235){\footnotesize $q$} \put(0,285){\footnotesize $q$}\put(0,335){\footnotesize $q$}
\put(-5,55){$q$} 
\put(-5,105){$q^2$} \put(-18,205){$q^{r-3}$}
\put(-18,255){$q^{r-2}$}\put(-17,305){$q^{r-1}$}\put(-5,355){$q^{r}$}

\put(-10,-5){$1_B$}
\put(60,-5){$p$} \put(110,-5){$p^2$} \put(160,-5){$p^3$} \put(260,-5){$p^{r-3}$} \put(310,-5){$p^{r-2}$}
 \put(360,-5){$p^{r-1}$} \put(410,-5){$p^{r}$}

\put(47,245){$q^{r-2}p$}\put(97,245){$q^{r-2}p^2$}
\put(47,295){$q^{r-1}p$}

\put(55,45){$qp$} \put(105,45){$qp^2$} \put(155,45){$qp^3$} \put(255,45){$qp^{r-3}$} \put(305,45){$qp^{r-2}$}
\put(355,45){$qp^{r-1}$}

\put(55,95){$q^2p$} \put(105,95){$q^2p^2$} \put(155,95){$q^2p^3$} \put(255,95){$q^2p^{r-3}$} \put(305,95){$q^2p^{r-2}$}
\put(55,195){$q^{r-3}p$} \put(105,195){$q^{r-3}p^2$} \put(155,195){$q^{r-3}p^3$}

\letvertex P=(32,28) \letvertex PP=(33,28) \drawedge(P,PP){}
\letvertex AP=(82,28) \letvertex APP=(83,28) \drawedge(AP,APP){}
\letvertex BP=(132,28) \letvertex BPP=(133,28) \drawedge(BP,BPP){}
\letvertex CP=(282,28) \letvertex CPP=(283,28) \drawedge(CP,CPP){}

\letvertex PA=(32,78) \letvertex PPA=(33,78) \drawedge(PA,PPA){}
\letvertex APA=(82,78) \letvertex APPA=(83,78) \drawedge(APA,APPA){}
\letvertex BPA=(132,78) \letvertex BPPA=(133,78) \drawedge(BPA,BPPA){}
\letvertex CPA=(282,78) \letvertex CPPA=(283,78) \drawedge(CPA,CPPA){}

\letvertex PB=(32,128) \letvertex PPB=(33,128) \drawedge(PB,PPB){}
\letvertex APB=(82,128) \letvertex APPB=(83,128) \drawedge(APB,APPB){}
\letvertex BPB=(132,128) \letvertex BPPB=(133,128) \drawedge(BPB,BPPB){}
\letvertex CPB=(282,128) \letvertex CPPB=(283,128) \drawedge(CPB,CPPB){}

\letvertex PC=(32,228) \letvertex PPC=(33,228) \drawedge(PC,PPC){}
\letvertex APC=(82,228) \letvertex APPC=(83,228) \drawedge(APC,APPC){}
\letvertex BPC=(132,228) \letvertex BPPC=(133,228) \drawedge(BPC,BPPC){}
\letvertex CPC=(282,228) \letvertex CPPC=(283,228) 

\put(30,35){\footnotesize $q$}\put(80,35){\footnotesize $q$}\put(130,35){\footnotesize $q$}
\put(280,35){\footnotesize $q$}

\put(30,17){\footnotesize $p$}\put(80,17){\footnotesize $p$}\put(130,17){\footnotesize $p$}
\put(280,17){\footnotesize $p$}

\put(30,85){\footnotesize $q$}\put(80,85){\footnotesize $q$}\put(130,85){\footnotesize $q$}
\put(280,35){\footnotesize $q$}

\put(30,67){\footnotesize $p$}\put(80,67){\footnotesize $p$}\put(130,67){\footnotesize $p$}
\put(280,17){\footnotesize $p$}

\put(30,135){\footnotesize $q$}\put(80,135){\footnotesize $q$}\put(130,135){\footnotesize $q$}
\put(280,135){\footnotesize $q$}

\put(30,117){\footnotesize $p$}\put(80,117){\footnotesize $p$}\put(130,117){\footnotesize $p$}
\put(280,117){\footnotesize $p$}

\put(30,235){\footnotesize $q$}\put(80,235){\footnotesize $q$}\put(130,235){\footnotesize $q$}

\put(30,217){\footnotesize $p$}\put(80,217){\footnotesize $p$}\put(130,217){\footnotesize $p$}

\drawundirectededge(A,AAAA){}
\drawundirectededge(B,BBBB){}
\drawundirectededge(C,CCCC){}
\drawundirectededge(D,DDDD){}
\drawundirectededge(A,C){}

\drawundirectededge(AAAAA,AAAAAA){}
\drawundirectededge(BBBBB,BBBBBB){}
\drawundirectededge(CCCCC,CCCCCC){}

\letvertex AB=(40,10)\letvertex AAB=(90,10) \letvertex AAAB=(140,10) \letvertex AAAAAB=(295,10)
\drawedge(A,AB){} \drawedge(AA,AAB){}\drawedge(AAA,AAAB){} \drawedge(AAAAA,AAAAAB){}

\letvertex ZAB=(40,60)\letvertex ZAAB=(90,60) \letvertex ZAAAB=(140,60) \letvertex ZAAAAAB=(295,60)
\drawedge(B,ZAB){} \drawedge(BB,ZAAB){}\drawedge(BBB,ZAAAB){} \drawedge(BBBBB,ZAAAAAB){}

\letvertex CAB=(40,110)\letvertex CAAB=(90,110) \letvertex CAAAB=(140,110) 
\letvertex CAAAAAB=(295,110)
\drawedge(C,CAB){} \drawedge(CC,CAAB){}\drawedge(CCC,CAAAB){} \drawedge(CCCCC,CAAAAAB){}

\letvertex DAB=(40,210)\letvertex DAAB=(90,210) \letvertex DAAAB=(140,210) 
\letvertex DAAAAAB=(295,210)
\drawedge(D,DAB){} \drawedge(DD,DAAB){}\drawedge(DDD,DAAAB){} 

\letvertex L=(400,10)
\letvertex LL=(400,60)
\letvertex LLL=(400,110)
\letvertex LLLL=(400,210)

\drawundirectedcurvededge(A,AA){}
\drawundirectedcurvededge(AA,AAA){}
\drawundirectedcurvededge(AAA,AAAA){}
\drawundirectedcurvededge(AAAAA,AAAAAA){}
\put(30,17){\footnotesize $p$}\put(80,17){\footnotesize $p$}\put(130,17){\footnotesize $p$}
\put(280,17){\footnotesize $p$}

\drawundirectedcurvededge(B,BB){}
\drawundirectedcurvededge(BB,BBB){}
\drawundirectedcurvededge(BBB,BBBB){}
\drawundirectedcurvededge(BBBBB,BBBBBB){}

\drawundirectedcurvededge(C,CC){}
\drawundirectedcurvededge(CC,CCC){}
\drawundirectedcurvededge(CCC,CCCC){}
\drawundirectedcurvededge(CCCCC,CCCCCC){}

\drawundirectedcurvededge(D,DD){}
\drawundirectedcurvededge(DD,DDD){}
\drawundirectedcurvededge(DDD,DDDD){}

\thinlines
\letvertex F=(10,300)
\drawundirectededge(A,DADA){} \thinlines
\thinlines
\drawundirectededge(A,AAAAAVA){}
\thinlines
\end{picture}
\end{center}
\caption{The ball $B_r(1_B)$ in the Cayley graph $\CC(B,\{p,q\})$ of the bicyclic monoid $B$}
\label{fig:BallB}
\end{figure}
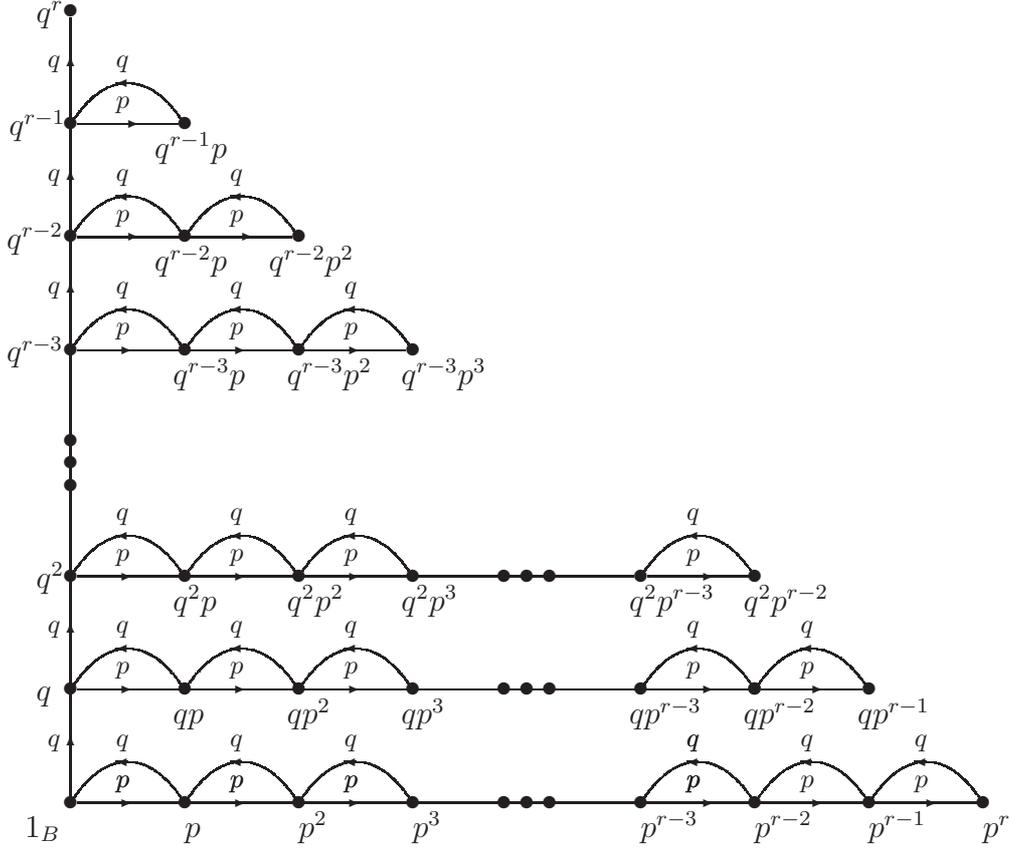
As the map $v \mapsto v'$ is clearly injective, we deduce that we must have  $|V(r)| \leq |V|/2$.
\end{example}

\section{Final remarks}
The same way we have the notions of left and right amenability, which coincide in the
group setting but are distinct in the more general setting of monoids, we may consider 
the notion of \emph{left-soficity} (resp. \emph{right-soficity} for monoids as follows.

First recall that the \emph{opposite monoid} of a monoid $(M,\cdot)$ is the monoid $(M^{op},\circ)$ with the same underlying set $M$ of elements and with multiplication 
defined by $a \circ b := b \cdot a$ for all $a,b \in M^{op}=M$.
Note that every abelian monoid $M$ is isomorphic to its opposite monoid (via the identity map $\Id_M$) and that every group $G$ is isomorphic to its opposite $G^{op}$ via the \emph{inversion map} $g \mapsto g^{-1}$.

Then, we say that a monoid $M$ is \emph{left-sofic} (resp. \emph{right-sofic}) if $M$ (resp. $M^{op}$) is ``sofic'' according to Definition \ref{d:sofic-semigroup}.
In other words, a monoid $M$ is right-sofic if it satisfies the following condition:
for every finite subset $K \subset M$ and every $\varepsilon > 0$,
there exist a non-empty finite set $X$ and a $(K,1 - \varepsilon)$-injective  $(K,\varepsilon)$-morphism $\varphi \colon M \to \Map(X)^{op}$.

Since every group $G$ (resp. abelian monoid $M$) is isomorphic to its opposite $G^{op}$ (resp. $M^{op}$) the two notions of left and right soficity coincide in the group (resp. abelian monoid) setting. Moreover, since finiteness (resp. residual finiteness, resp. cancellative one-side amenability) are preserved under the opposite monoid operation $M \mapsto M^{op}$, we deduce that left and right soficity coincide for finite (resp. residually finite, resp. cancellative one-sided amenable) monoids by virtue of Proposition \ref{t:finite-are-sofic} (resp. Corollary \ref{c:residually-sofic}, resp. Proposition \ref{p:canc-amen-are-sofic}).
We don't know, however, whether or not these two notions also coincide for general 
monoids.

Note that the bicyclic monoid $B = \langle p,q : pq = 1 \rangle$ is isomorphic to
its opposite $B^{op}$ via the map exchanging the generators $p$ and $q$ so that, by virtue of
Theorem \ref{t:bicyclic-not-sofic} $B$ is neither left nor right sofic.

As finitely generated monoids are concerned, we remark the following. Let $M$ be a finitely
generated monoid and $\Sigma$ a finitely generating subset of $M$. Let us call the
graph $\CC(M, \Sigma) = (V,E)$ (resp. $\CC'(M, \Sigma) = (V',E')$) with $V = V' = M$ and $E = \{(m,\sigma,m\sigma): m \in M, \sigma \in \Sigma\}$ (resp. $E' = \{(m,\sigma,\sigma m): m \in M, \sigma \in \Sigma\}$ the  \emph{left Cayley graph} (resp. \emph{right Cayley graph}) of $M$ with respect to $\Sigma$.
It is then clear that $\CC'(M,\Sigma)$ (resp. $\CC'(M^{op},\Sigma)$  is label isomorphic to $\CC(M^{op},\Sigma)$ (resp. $\CC(M,\Sigma)$). All this said, we have the following characterization of right-soficity for finitely generated monoids (cf. Theorem \ref{t:weiss-characterization}).
A finitely generated monoid $M$ is right-sofic if and only if for every $r \in \N$ and every $\delta > 0$, there exists a finite $\Sigma$-labeled graph $\GG' = (V,E)$ with the following property: the subset $V(r) \subset V$, consisting of all the vertices $v \in V$ such that the ball of radius $r$ centered at $v$ in $\GG'$ is isomorphic, as a pointed $\Sigma$-labeled graph, to the ball of radius $r$ centered at $1_M$ in the right Cayley graph $\CC'(M,\Sigma)$, satisfies \eqref{eq:condition-weiss}.\\

\noindent
{\bf Acknowledgments.} We express our deepest gratitude to Jan Cannizzo, Fabrice Krieger and Liviu Paunescu for useful comments and remarks. 

\def\cprime{$'$} \def\cprime{$'$}

\end{document}